\documentclass[reqno,12pt,letterpaper]{amsart}
\usepackage{amsmath,amssymb,amsthm,graphicx,mathrsfs,url}
\usepackage{amsfonts,bm,bbm}
\usepackage[usenames,dvipsnames]{color}
\usepackage[colorlinks=true,linkcolor=Red,citecolor=Green]{hyperref}

\def\?[#1]{\textbf{[#1]}\marginpar{\Large{\textbf{??}}}}

\graphicspath{{./}{./figures/}}

\newtheorem{theo}{Theorem}
\newtheorem{prop}{Proposition}[section]

\newtheorem{lemm}[prop]{Lemma}

\newtheorem{rem}{Remark}
\numberwithin{equation}{section}

\newcommand{\mc}{\mathcal}
\newcommand{\rr}{\mathbb{R}}
\newcommand{\nn}{\mathbb{N}}
\newcommand{\cc}{\mathbb{C}}
\newcommand{\hh}{\mathbb{H}}
\newcommand{\zz}{\mathbb{Z}}
\newcommand{\sph}{\mathbb{S}}

\newcommand{\dd}{\mathbb{D}}

\newcommand{\la}{\lambda}
\newcommand{\eps}{\epsilon}

\newcommand{\pl}{\partial}
\newcommand{\x}{\times}

\newcommand{\til}{\widetilde}
\newcommand{\bbar}{\overline}

\newcommand{\cjd}{\rangle}
\newcommand{\cjg}{\langle}

\newcommand{\demi}{\tfrac{1}{2}}

\newcommand{\rmx}{\mathrm{x}}
\newcommand{\bfe}{\mathbf{e}}

\def\indic{\operatorname{1\hskip-2.75pt\relax l}}

\title{Reconstruction formulas for X-ray transforms in negative curvature}
\author{Colin Guillarmou}
\email{cguillar@dma.ens.fr}
\address{DMA, U.M.R. 8553 CNRS, \'Ecole Normale Superieure, 45 rue d'Ulm,
75230 Paris cedex 05, France}
\author{Fran\c cois Monard}
\email{monard@umich.edu}
\address{Dept. of Mathematics, 2074 East Hall. 530 Church St., University of Michigan, 
Ann Arbor, MI 48109-1043, USA} 

\begin{document}
\maketitle

\begin{abstract}
  We give reconstruction formulas inverting the geodesic X-ray transform over functions (call it $I_0$) and solenoidal vector fields on surfaces with negative curvature and strictly convex boundary. These formulas generalize the Pestov-Uhlmann formulas in \cite{PeUh} (established for simple surfaces) to cases allowing geodesics with infinite length on surfaces with trapping. Such formulas take the form of Fredholm equations, where the analysis of error operators requires deriving new estimates for the normal operator $\Pi_0 = I_0^* I_0$. Numerical examples are provided at the end. 
\end{abstract}

\section{Introduction}
For a Riemannian metric $g$ on a manifold $M$ with strictly convex boundary $\pl M$, 
the X-ray transform $I_0f$ of a function $f\in C^0(M)$ 
is the collection of integrals over all finite length geodesics with endpoints on $\pl M$.
For a domain in $\rr^2$ equipped with the Euclidean metric, this is nothing more than 
the Radon transform. A natural inverse problem is to recover $f$ from $I_0f$, ie. to find an inversion procedure for $I_0$. In the Euclidean case, Radon found an inversion formula \cite{Ra}, and this has been extended to hyperbolic spaces by Helgason \cite{He} and Berenstein-Casadio Tarabusi \cite{BeCa}. 

This problem has been studied more generally for \emph{simple metrics}, which are metrics with no conjugate points and such that each geodesic has finite length.  Let $SM=\{(x,v)\in TM; |v|_{g_x}=1\}$ be the unit tangent bundle of $M$ and $\pi_0:SM\to M$ the canonical projection $(x,v)\mapsto x$; if $\nu$ denotes the inward unit pointing vector field to $\pl M$, we define $\pl_\mp SM=\{(x,v)\in SM; \pm \cjg v,\nu\cjd_g>0\}$ the inward (-) and outward (+) boundaries. Then 
if $\varphi_t$ denotes the geodesic flow on $SM$ at time $t$ and if $g$ is simple, we can 
view $I_0$ as the map 
\begin{equation}\label{defI0intro}
I_0 : C^0(M)\to C^0(\partial_-SM), \quad I_0f(x,v)=\int_{0}^{\ell(x,v)}f(\pi_0(\varphi_t(x,v)))dt
\end{equation}
where $\ell(x,v)$ is the length of the geodesic $\pi_0(\varphi_{t}(x,v))$ relating boundary points of 
$M$. The injectivity of $I_0$ for simple metrics has been proved by Mukhometov \cite{Mu} and Anikonov \cite{An} (see also the book by Sharafutdinov \cite{Sh}). More recently Pestov-Uhlmann \cite{PeUh} gave an approximate inversion formula in dimension $2$ for simple metrics, which is exact when the curvature of $g$ is constant. Krishnan \cite{Kr} then showed that the Pestov-Uhlmann formula can be made exact in a small enough $C^3$-neighborhood of a given simple metric $g_0$ with constant curvature, but the bound is not explicit and depends for instance on the diameter of $(M,g_0)$. The Pestov-Uhlmann approach, exploiting the interplay between transport equations on $SM$ and the fiberwise Hilbert transform, was used by the second author to derive further inversion formulas for certain types of weighted X-ray transforms on simple surfaces \cite{Mo2,Mo3}.

Manifolds with geodesics having infinite length are called \emph{manifolds with trapping}. For instance any negatively curved Riemannian manifold with strictly convex boundary which have non contractible loops have closed geodesics and thus geodesics with infinite lengths. In general we call $\Gamma_\pm\subset SM$ the set 
of points $(x,v)$ such that $\varphi_{\mp t}(x,v)$ belongs to the interior $SM^\circ$ of $SM$ for all $t>0$, which means that the geodesic $\cup_{t>0}\pi_0(\varphi_{\mp t}(x,v))$ has infinite length. The trapped set $K\subset SM^\circ$ is defined by $\Gamma_-\cap \Gamma_+$ and consists of  
geodesics never touching the boundary. When $g$ has negative curvature,
 $K$ is a hyperbolic set for the flow in the sense of \eqref{hyperbolicK} and $g$ has no conjugate points.
In the recent work \cite{Gu2}, the first author proved that for a class of manifolds 
including those with negative curvature (see \eqref{assumptions} for the precise assumptions), the $X$-ray transform $I_0f$  can be defined by the expression \eqref{defI0intro}� outside 
$\Gamma_-\cap \pl_-SM$ and $I_0$ extends as a bounded and injective operator 
\[ I_0 : L^p(M) \to L^2(\pl_-SM,d\mu_\nu), \quad \forall p>2,\]
where $d\mu_\nu$ is some smooth Lebesgue type measure on $\pl_-SM$ 
multiplied by a weight vanishing on $S\pl M\subset \pl SM$. In fact, the same result holds as a map $L^2(M)\to L^2(\pl SM)\subset L^2(\pl_-SM,d\mu_\nu)$ by the boundedness proved in Lemma \ref{I_0I_1} below.
For $f\in L^2(SM)$, the odd extension of $I_0f$ is defined as the $L^2(SM)$ element 
\[I_0^{\rm od}f(x,v):=\left\{\begin{array}{ll}
I_0f(x,v) ,  & (x,v)\in \pl_-SM \\
-I_0f(x,-v), & (x,v)\in \pl_+SM
\end{array}\right. .\] 

Similarly, we can integrate $1$-forms along geodesics and we can define the X-ray transform on $1$-forms by 
\[I_1f(x,v):=\int_{0}^{\ell(x,v)}(\pi_1^*f)(\varphi_t(x,v))dt, \quad (x,v)\in \pl_-SM\setminus \Gamma_-\]
where $\pi_1^*: C^\infty(M,T^*M)\to C^\infty(SM)$ is defined by $\pi_1^*f(x,v)=f(x)(v)$, this also extends as  
a bounded operator $I_1 : L^2(M,T^*M) \to L^2(\pl_-SM)\subset L^2(\pl_-SM,d\mu_\nu)$ which is injective on 
the space of divergence free forms.
We let $I_1^*$ be the adjoint operator when $I_1$ is viewed as a map $L^2(M,T^*M) \to L^2(\pl_-SM,d\mu_\nu)$. 

To state our result we need to define two more operators. 
Let $H$ be the Hilbert transform in the fibers of $\pl SM$ (i.e. multiplication by $-i{\rm sign}(k)$ 
on the $k$ Fourier component in the fiber, with ${\rm sign}(0):=0$), and let 
$\mc{S}_g: L^2(\pl_- SM)\to L^2(\pl_+SM)$ be the scattering operator of the flow, defined by 
\[ \mc{S}_gf(x,v):= f(\varphi_{-\ell(x,v)}(x,v)),\quad  (x,v)\notin \Gamma_+.\]

Then we have the following approximate inversion formula extending Pestov-Uhlmann formula to the  setting of manifolds with trapping.
\begin{theo}\label{inversionintro}
Let $(M,g)$ be a Riemannian surface with strictly convex boundary and assume that the trapped set $K$ for the flow is hyperbolic, that $g$ has no conjugate points (these conditions are satisfied in negative curvature). Then there exists an operator $W:L^2(M)\to L^2(M)$ with smooth integral kernel such that for each $f\in L^2(M)$
\begin{equation*}
f+W^2f= -\frac{1}{4\pi} *dI_1^*\mc{S}_g^{-1}(HI_0^{\rm od}f)|_{\pl_+SM},
\end{equation*} 
where $*$ is the Hodge star operator and $d$ the exterior derivative. \end{theo}
We also have an inversion formula for divergence free $1$-forms, see Theorem \ref{inversion}.
We next prove a upper bound on the $L^2\to L^2$ norm of $W$ in a neighborhood of constant negatively curved metrics, showing that $({\rm Id}+W^2)$ can be inverted by the Neumann series, and thus providing an exact inversion formula for $I_0$. In constant curvature, the trapped set $K$ is a fractal set with 
Hausdorff dimension $\dim_{{\rm Haus}}(K) \in [1,3)$ when $K\not=\emptyset$ (see \cite{Pat,Su}).
\begin{theo}\label{Theo2}
Let $(M,g_0)$ be a manifold with strictly convex boundary and constant negative curvature $-\kappa_0$ and trapped set $K$, and let $\delta=\demi(\dim_{{\rm Haus}}(K)-1)\in [0,1)$. 
Then for each $\la_1,\la_2\in (0,1)$ so that 
$1\geq \la_1\la_2>\max(\delta,\demi)$, there is a an explicit constant $A(\delta,\la_1,\la_2)$ depending only on $\delta,\la_1,\la_2$
such that for all metric $g$ on $M$ with strictly convex boundary and Gauss curvature $\kappa(x)$ satisfying 
\[ \la_1^{2}g_0\leq g\leq \la_1^{-2} g_0, \quad  \kappa(x) \leq -\la_2^2\kappa_0, \quad 
|d\kappa|_{L^\infty}\leq A(\delta,\la_1,\la_2)\kappa_0^{3/2}\]
the remainder operator $W$ obtained in Theorem \ref{inversionintro} has norm 
$||W||_{L^2(M,g)\to L^2(M,g)}<1$ and thus the inversion formula 
\eqref{formulainv0} allows to recover $f\in L^2(M)$ from $I_0f$ by a convergent series. When $\delta<1/2$, the constant $A(\delta,\la_1,\la_2)$ does not depend on $\delta$.
\end{theo}
The constant $A(\delta,\la_1,\la_2)$ is given by $A(\delta,\la_1,\la_2)=3\la_1^4/(\la_2C(\delta,\la))$
where $1-\la=\la_1\la_2$ and $C(\delta,\la)$ is the constant given in \ref{bound1}. This shows that there is an explicit neighborhood of the space of constant curvature metrics where the inversion formula can be made exact. Even when $K=\emptyset$ the result is new since our bounds on the size of the neighborhood of constant curvature metrics is independent of the volume and the diameter of $(M,g)$, which was not the case in the previous works for simple metrics \cite{Kr,Mo2}. To prove such a bound, we compute on convex co-compact hyperbolic surfaces an operator closely related to $(I_{0}^\la)^*I_0^{\la}$ in terms of the Laplacian on the surface, where $I_0^\la$ is the X-ray transform with a constant attenuation $\la\in [0,\min(\demi,1-\delta))$. This is done in the Appendix.
Finally, we remark that is is quite possible that Helgason type inversion for $I_0$ on $\hh^2$ could be applied to deal with the case of  constant curvature metrics (passing to the universal cover), but that would involve checking that the inversion formula applies to functions which are periodic by a discrete group, thus not tending to $0$ at infinity of $\hh^2$, and in any case this would not work in variable curvature, unlike our formula.

\textbf{Acknowledgement.} C.G. is partially supported by ANR grant ANR-13-JS01-0006 and ANR-13-BS01-0007-01. FM is partially supported by NSF grant DMS-1514820. Part of this work was done during the research program Inverse Problems at IHP, Paris, Spring 2015.

\section{Preliminaries} 
\subsection{Geometry of $SM$ and geodesic flow}
Let $(M,g)$ be an oriented surface with strictly convex boundary and let 
$SM$ be its unit tangent bundle with canonical projection $\pi_0: SM\to M$. 
We will assume that 
\begin{equation}\label{assumptions}
\begin{split}
1) & \,\, g \textrm{ has no conjugate points }\\
2) & \,\, K \textrm{ is a hyperbolic set for the flow}
%3) & \,\, \textrm{the periodic orbits are dense in }K. 
\end{split}\end{equation}
Notice that these assumptions are satisfied if $g$ has negative curvature, see \cite{Gu2}. 
We denote by $X$ the geodesic vector field on $SM$ and $\varphi_t$ its flow at time $t$.
Recall that hyperbolicity of the flow means that there is a continuous, flow invariant, decomposition on 
$K$
\[\forall y\in K, \quad  T_y(SM)=\rr X(y)\oplus E_s(y)\oplus E_u(y)\] 
where $E_s(y)$ and $E_u(y)$ are the stable and unstable subspaces satisfying for each $y\in K$
\begin{equation}\label{hyperbolicK} 
\begin{gathered}
||d\varphi_t(y)w||\leq Ce^{-\gamma |t|}||w||,\quad  \forall t>0, \forall w\in E_s(y),\\
||d\varphi_t(y)w||\leq Ce^{-\gamma |t|}||w||,\quad  \forall t<0, \forall w\in E_u(y).
\end{gathered}
\end{equation} 
some uniform $C,\gamma>0$.
Let $V$ be the vertical vector field on $SM$ defined by $Vf(x,v)=\pl_\theta(f(R_\theta(x,v)))|_{\theta=0}$ where $R_\theta(x,v)$ is the $+\theta$ rotation in the fibers of $SM$.  Locally, in isothermal coordinates $(x_1,x_2)$ on $M$, the metric can be written as 
$g=e^{2\phi}(dx_1^2+dx_2^2)$ for some smooth function $\phi$, 
we have associated coordinates $(x_1,x_2, \theta)$ on $SM$ if $\theta\in S^1$
parametrizes the fibers elements by $v=e^{-\phi}(\cos(\theta)\pl_{x_1}+\sin(\theta)\pl_{x_2})$; 
in these coordinates $V$ is simply given by $\pl_\theta$.  
The tangent space $TSM$ of $SM$ is spanned by the vector fields 
$X,V,X_\perp$ where $X_\perp:=[X,V]$; the vectors $X,X_\perp$ span the horizontal space $\mc{H}$, while $V$ spans the vertical space $\mc{V}$ in $TSM$. The Sasaki metric $G$ on $SM$ is the metric such that $(X,X_\perp,V)$ 
is orthonormal. The Liouville measure $\mu$ on $SM$ is the Riemannian measure of the metric $G$ and
we denote by $\mu_{\pl SM}$ the measure induced by $G$ on the boundary $\pl SM$ of $SM$. We define 
$L^p(SM)$ and $L^p(\pl SM)$ the $L^p$ spaces with respect to $\mu$ and $\mu_{\pl SM}$. There is another natural measure $\mu_\nu$ on $\pl SM$ which is given by 
\begin{equation}\label{dmunu} 
d\mu_{\nu}(x,v):= |\cjg v,\nu\cjd_g|\, d\mu_{\pl SM}(x,v).
\end{equation}
where $\nu$ is the inward pointing normal vector field to $\pl M$. Let us also define the incoming (-), 
outgoing (+) and glancing (0) boundaries of $SM$ 
\[ \begin{gathered}
\pl_\mp SM:=\{ (x,v)\in \pl SM;   \pm \cjg \nu,v\cjd_g > 0\},\quad 
\pl_0SM=\{ (x,v)\in \pl SM;  \cjg\nu,v\cjd_g=0\}.
\end{gathered}\]

Define $(M_e,g)$ to be a small extension of $M$ with no conjugate points and with strictly convex boundary, and so that each geodesic with initial point $(x,v)\in \pl_+SM\cup \pl_0SM$ goes in finite time to $\pl SM_e$; 
the existence of $M_e$ is proved in Section 2.1 and Lemma 2.3 of \cite{Gu2}. We keep the notation $X$ for 
the vector field of the geodesic flow on $SM_e$. The Liouville measure $\mu$ extends to $SM_e$. 
For each point $(x,v)\in SM$, define the time of escape of $SM$ in positive (+) and negative (-) time:
\begin{equation}\label{ellpm} 
\begin{gathered}
\ell_+ (x,v)=\sup\, \{ t\geq 0; \varphi_{t}(x,v)\in SM\}\subset [0,+\infty],\\
\ell_- (x,v)=\inf\, \{ t\leq 0; \varphi_{t}(x,v)\in SM\}\subset [-\infty,0]
\end{gathered}
\end{equation}
and the incoming (-) and outgoing (+) tail are defined by 
\[ \Gamma_-= \{(x,v)\in SM; \ell_+(x,v)=\infty\} ,\quad \Gamma_+= \{(x,v)\in SM; \ell_-(x,v)=-\infty\}.\]
We define $\Gamma_\pm$ as well on $SM_e$ in the same way, so that they extend those on $SM$,
and they are closed set in $SM_e$ are we keep the same notation $\Gamma_\pm$ for both $SM$ and $SM_e$. 
The trapped set is $K=\Gamma_-\cap \Gamma_+\subset SM^\circ$ and is invariant by the flow.
We can define the extended stable bundle $E_-\subset T_{\Gamma_-}SM_e$ 
over $\Gamma_-$ satisfying $E_-|_{K}=E_s$ and 
\[\forall y \in \Gamma_-, \, \forall t>0, \, \forall w\in E_-(y),\,\,  ||d\varphi_t(y)w||\leq Ce^{-\gamma |t|}||w||\]
for some uniform $C,\gamma>0$; the extended unstable bundle $E_+$ over $\Gamma_+$ is defined similarly by taking  negative times instead of $t>0$.
The bundles $E_-^*\subset T^*_{\Gamma_-}SM_e$ and $E_+^*\subset T^*_{\Gamma_+}SM_e$ 
are defined by 
\begin{equation}\label{E_-*} 
E_-^*(E_-\oplus \rr X)=0, \quad E_+^*(E_+\oplus \rr X)=0.
\end{equation}
It follows from \cite{BoRu} that under the assumptions \eqref{assumptions} 
we have (see \cite[Prop. 2.4]{Gu2} for details)
\[{\rm Vol}_{\mu}(\Gamma_-\cup\Gamma_+)=0, \quad {\rm Vol}_{\mu_\nu}(\pl SM\cap(\Gamma_-\cup\Gamma_+))=0.\] 

\subsection{Hilbert transform and Pestov-Uhlmann relation}
Following  \cite{GK1}, there is an orthogonal decomposition corresponding to  Fourier series in the fibers  
\begin{equation}\label{decomp}
L^2(SM_e)=\bigoplus_{k\in\zz} \Omega_k(M_e), \quad \textrm{ with }Vw_k=ikw_k \textrm{ if }w_k\in \Omega_k(M_e)
\end{equation}
where $\Omega_k$ is the space of $L^2$ sections of a complex line bundle over $M_e^\circ$. 
Similarly, one has a decomposition on $\pl SM$
\begin{equation}\label{decompplSM}
L^2(\pl SM)=\bigoplus_{k\in\zz} \Omega_k(\pl SM), \quad \textrm{ with }V\omega_k=ik\omega_k \textrm{ if }\omega_k\in \Omega_{k}(\pl SM)
\end{equation}
using Fourier analysis in the fibers of the circle bundle. There is a canonical map 
\[ \pi_m^*: C^\infty(M_e,\otimes_S^mT^*M_e)\to C^\infty(SM_e), \quad 
(\pi_m^*u)(x,v):=u(x)(\otimes^m v)\]
and in particular $\pi_0^*$ is just the pull-back by the projection $\pi_0: SM_e\to M_e$. We define the dual map
${\pi_m}_*$ on distributions $C^{-\infty}(SM_e^\circ)$ 
by $\cjg {\pi_m}_*u,f\cjd=\cjg u,\pi_m^*f\cjd$ where the distribution pairing is $\cjg u,\psi\cjd=\int_{SM_e}u\psi d\mu$ when $u\in L^2(SM^\circ_e)$.
We notice that $\Omega_m\oplus \Omega_{-m}=\pi_m^*{\pi_m}_*(L^2(SM_e))$ and it is easily checked that 
\[ w\mapsto \frac{1}{2\pi}\pi_0^*{\pi_0}_*w, \quad\quad  w\mapsto 
\frac{1}{\pi}\pi_1^*{\pi_1}_*w\]
are the orthogonal projection from $L^2(SM_e)$ to respectively $\Omega_0$ and $\Omega_1\oplus \Omega_{-1}$.

The Hilbert transform in the fibers is defined by using the decomposition \eqref{decomp}: 
\[H: C_c^\infty(SM^\circ_e)\to C_c^\infty (SM^\circ_e),\quad H(\sum_{k\in\zz}w_k):=-i\sum_{k\in\zz}{\rm sign}(k)w_k.\]
with ${\rm sign}(0):=0$ by convention. We can extend continuously 
$H$ to $C^{-\infty}(SM^\circ_e)\to C^{-\infty}(SM^\circ_e)$ by the expression
\[ \cjg Hu,\psi\cjd :=-\cjg u,H\psi\cjd , \quad \psi\in C_c^\infty(SM^\circ_e).\]
Similarly, we define the Hilbert transform in the fibers on $\pl SM$
\[H_\pl: C^\infty(\pl SM)\to C^\infty(\pl SM),\quad H_\pl(\sum_{k\in\zz}\omega_k)=-i\sum_{k\in\zz}{\rm sign}(k)\omega_k\]
and its extension to distributions as for $SM_e$. We note that $H$ extends as 
a bounded operator on $L^2(SM_e)$ and $H_\pl$ as a bounded operator on $L^2(\pl SM)$. 
For smooth $w\in C_c^\infty(SM_e^\circ)$ we have that 
\begin{equation}\label{hilbertrest} 
(Hw)|_{\pl SM}=H_{\pl}\,\omega, \quad\textrm{ with } \omega:=w|_{\pl SM}
\end{equation}
and the identity extends by continuity to the space of distributions in $SM_e^\circ$ with wave-front set
disjoint from $N^*(\pl SM)$ since, by \cite[Th. 8.2.4]{Ho}, the restriction map $C^\infty(SM_e^\circ)\to C^\infty(\pl SM)$ 
obtained by pull-back through the inclusion map $\iota:\pl SM\to SM_e^\circ$ extends continuously to the space 
of distributions on $SM_e^\circ$ with wavefront set not intersecting $N^*(\pl SM)$.
The following relation between Hilbert transform and flow was proved by Pestov-Uhlmann \cite[Th. 1.5]{PeUh2} (it holds on any surface): 
\begin{equation}\label{commutator}
\textrm{ if }w\in C^\infty(SM_e^\circ), \quad [H,X]w=X_\perp w_0+(X_\perp w)_0
\end{equation}
where $w_0:=\frac{1}{2\pi}\pi_0^*{\pi_0}_*w$ for $w\in C^{-\infty}(SM_e^\circ)$. 
We can use the odd/even decomposition of distributions with respect to the involution $A(x,v)=(x,-v)$ on 
$SM_e$, $SM$ and $\pl SM$ : for instance $X$ maps odd distributions to even distributions and conversely, 
$H$ maps odd (resp. even) distributions to odd (resp. even) distributions. We set 
$H_{\rm ev} w:=H(w_{\rm ev})$ and $H_{\rm od}w:=H(w_{\rm od})$. We write similarly 
$H_{\pl,{\rm ev}}$ and $H_{\pl,{\rm od}}$ for the Hilbert transform on (open sets of) $\pl SM$ and the relation \eqref{hilbertrest} also holds with $H_{\pl, {\rm ev}}$ replacing $H_{\pl}$ if $w$ is even.
Taking the odd part of \eqref{commutator}, for any $w\in C^{-\infty}(SM_e^\circ)$ 
\begin{equation}\label{oddpart}
\begin{gathered}
H_{\rm ev}Xw-XH_{\rm od}w=(X_\perp w)_0.
\end{gathered}
\end{equation}

\subsection{Boundary value problem and X-ray transform}
For convenience, we extend $M_e$ to a larger with boundary manifold $\hat{M}$ and extend smoothly $X$ to $S\hat{M}$ in a way that the flow of $X$ in $S\hat{M}$ is complete 
and for each $y\in SM_e^0$, if $\varphi_{t_0}(y)\in \pl SM_e$ for some $t_0 >0$, then 
$\varphi_t(y)\in S\hat{M}$ for all $t>t_0$; we refer to \cite[Section 2.1]{Gu2}. Let $\mu$ be the Liouville measure on $M_e$, which we extend smoothly to $S\hat{M}$ in any fashion. The invariance of Liouville measure by the flow in $SM_e$ is satisfied, ie. $\mc{L}_X\mu=0$, so that $X$ is formally skew-adjoint operator when acting on 
$C^\infty_c(SM_e^\circ)$. 

By \cite[Section 4.2]{Gu2}, under the assumptions \eqref{assumptions}, for each $f\in C^0(SM)$, the boundary value problem 
\[ Xu=-f, \quad u|_{\pl_+SM}=0\] 
has a unique solution in $L^1(SM)\cap C^0(SM\setminus \Gamma_-)$ given by $u=R_+f$ with 
\begin{equation}\label{R+} 
R_+f(x,v):= \int_{0}^\infty f(\varphi_t(x,v))dt.
\end{equation}
The operator $R_+: C^0(SM)\to L^2(SM)$ defined by \eqref{R+}� is bounded and the expression \eqref{R+} actually extends as a  bounded map $C^0(SM_e)\to L^2(SM_e)$ as well, see \cite[Proposition 4.2]{Gu2}. 
It has the following mapping properties 
\begin{equation}\label{mappingR+}
\begin{gathered}
R_+ : C^\infty(SM)\to C^\infty(SM\setminus (\Gamma_-\cup \pl_0SM)), \\
R_+: C_c^\infty(SM_e^\circ)\to C^\infty(SM_e\setminus \Gamma_-) \textrm{ and }
{\rm WF}(R_+f)\subset E_-^*  \textrm{ if }f\in C_c^\infty(SM_e^\circ),\\
R_+: H_0^s(SM_e)\to H^{-s}(SM_e) , \quad \forall s\in (0,1/2)
\end{gathered}
\end{equation}
where $E_-^*$ is defined by \eqref{E_-*}.

Define the \emph{scattering map} $S_g$ by 
\[ S_g: \pl_-SM\setminus \Gamma_-\to \pl_+SM\setminus \Gamma_+ , \quad 
S_g(x,v):= \varphi_{\ell_+(x,v)}(x,v)\]
By \cite[Lemma 3.4]{Gu2} the following operator is unitary 
\begin{equation}\label{unitary}
\mc{S}_g: L^2(\pl_-SM,d\mu_\nu)\to L^2(\pl_+SM,d\mu_\nu), \quad \mc{S}_g\omega:= \omega\circ S_g^{-1}.
\end{equation}

The \emph{X-ray transform} $I: C^0(SM)\to C^0(\pl_-SM\setminus \Gamma_-)$ is defined by 
\begin{equation}\label{defXray} 
If:= (R_+f)|_{\pl_-SM}.
\end{equation}
From \cite[Lemma 5.1]{Gu2} and \cite[Prop. 2.4]{Gu2},  
we have that, under the assumptions \eqref{assumptions}, it extends as a bounded operator 
\[ I: L^p(SM)\to L^2(\pl_-SM,d\mu_\nu) \quad \forall p>2.\]
We denote by $I^{\rm od}: L^p(SM)\to L^2(\pl SM,d\mu_\nu)$ its odd continuation defined by 
\begin{equation}\label{Iod}
I^{\rm od}f(x,v):=\left\{\begin{array}{ll}
If(x,v) ,  & (x,v)\in \pl_-SM \\
-If(x,-v), & (x,v)\in \pl_+SM
\end{array}\right.
\end{equation} 
and similarly we define the even continuation by $I^{\rm ev}f(x,v)=If(x,-v)$ for $(x,v)\in \pl_+SM$.
The operator $I^*: L^2(\pl_-SM,d\mu_\nu)\to L^{p'}(SM)$ for all $p'<2$ is the adjoint of $I$, it is obtained as follows: for $\omega\in L^2(\pl_-SM,d\mu_\nu)$,  $I^*\omega$ is the unique $L^1$ solution of 
\[ Xw= 0 , \quad w|_{\pl_-SM}=\omega.\]  
By this we mean that $w$ has an extension in a small neighborhood $U$ of $SM\setminus \pl_0SM$ as an invariant distribution and the restriction $w|_{\pl_-SM}$ makes sense as a distribution since, 
by elliptic regularity, $w$ has wave-front set ${\rm WF}(w)\subset \{\xi \in TU; \xi(X)=0\}$ which does not intersect the conormal $N^*(\pl_-SM)$, as $\rr X\oplus T\pl_-SM=TSM$ over $\pl_-SM$.
A consequence of this is that the unique $L^1$ solution of $Xu=-F$, $u|_{\pl_+SM}=\omega$ with $F\in C^\infty(SM)$ and $\omega\in L^2(\pl_+SM, d\mu_\nu)$ is given by $u=R_+F+I^*\mc{S}_g^{-1}\omega$

Since it will be convenient to apply the Hilbert transform on $L^2(\pl SM)$, we first show 
\begin{lemm}\label{L^2}
Under the assumptions \eqref{assumptions}, the X-ray transform is bounded as a map
\[ I^{\rm od}: L^p(SM)\to L^2(\pl SM), \quad \forall p>2.\]
There exists $\eps>0$ small such that the following operator is bounded
\[\indic_{\{|\cjg v,\nu\cjd|<\eps\}}I^{\rm od}: L^2(SM)\to L^2(\pl SM).\] 
\end{lemm} 
\begin{proof} It suffices to consider $(I^{\rm od}f)|_{\pl _-SM}=If$ as the part on $\pl_+SM$ 
is clearly the same. 
By Proposition 2.4 in \cite{Gu2}, we have $V(t)=\mc{O}(e^{Qt})$ for some $Q<0$ if 
\[V(t):={\rm Vol}\{ (x,v)\in SM; \ell_+(x,v)>t\}.\]
This directly implies that for all $1\leq p<\infty$ , we have 
$\ell_+\in L^p(\pl SM)$. By strict convexity of $\pl M$, there is $\eps>0$ small such that
if $|\cjg \nu,v\cjd|\leq \eps$, one has $\ell_+(x,v)\leq C|\cjg \nu,v\cjd|$ 
for some uniform $C>0$, see \cite[Lemma 4.1.2]{Sh} for example.
We then get 
\begin{equation}\label{inLp}
(x,v)\mapsto \frac{\ell_+(x,v)}{|\cjg \nu,v\cjd|} \in L^p(\pl SM)\,\, \textrm{ for all }p<\infty.
\end{equation}
Using H\"older inequality and Santalo formula, we get for 
$f\in C^\infty(SM)$ and $p>2$, if $1/p'+1/p=1$ and $r:=p'(p-1)/(p-2)$
\[ \begin{split} 
||If||^2_{L^2(\pl_-SM)} \leq &
\int_{\pl_-SM}\Big(\int_{0}^{\ell_+(y)}|f(\varphi_t(y))|^{p}dt\Big)^{\frac{2}{p}}\ell_+(y)^{\frac{2}{p'}}d\mu_{\pl SM}(y)
\\
\leq & \int_{\pl_-SM}\Big(\int_{0}^{\ell_+(x,v)}|f(\varphi_t(x,v))|^{p}dt|\cjg\nu,v\cjd|\Big)^{\frac{2}{p}}
\frac{\ell_+(x,v)^{\frac{2}{p'}}}{|\cjg\nu,v\cjd|^{\frac{2}{p}}}d\mu_{\pl SM}(x,v)
\\
\leq & \Big(\int_{\pl_-SM}\int_{0}^{\ell_+(y)}|f(\varphi_t(y))|^{p}dtd\mu_{\nu}(y)\Big)^{\frac{2}{p}} 
\Big(\int_{\pl_-SM}\Big|\frac{\ell_+(x,v)}{\cjg\nu,v\cjd}\Big|^{\frac{2r}{p'}} d\mu_{\pl SM}\Big)^{\frac{1}{r}} \\
||If||^2_{L^2(\pl_-SM)}  \leq & C_0||f||_{L^p(SM)}^{\frac{2}{p}} 
\end{split}\]
where $C_0<\infty$. This concludes the proof of the first statement. The boundedness of 
$\indic_{\{|\cjg v,\nu\cjd|<\eps\}}I^{\rm od}$ on $L^2$ is direct from the proof above by taking $p=2$ and using 
$\ell_+(x,v)\leq C|\cjg v,\nu\cjd|$ on $\pl_-SM$ if $|\cjg v,\nu\cjd|\leq \eps$ for $\eps>0$ small enough.
\end{proof}

Finally the X-ray transform on functions $I_0$ and on 1-forms $I_1$ are defined as the bounded operators, for 
any $p>2$,
\begin{equation} \label{I_0I_1}
I_0:= I\pi_0^* : L^p(M)\to L^2(\pl_-SM,d\mu_\nu), \quad I_1:= I\pi_1^* : L^p(M;T^*M)\to L^2(\pl_-SM,d\mu_\nu).
\end{equation}
and $I_j^{\rm od}=I^{\rm od}\pi_j^*$ for $j=0,1$. In fact, we can show 
\begin{lemm}\label{I_0L^2}
Under the assumptions \eqref{assumptions}, the X-ray transform on functions and on $1$-forms are bounded as maps
\[ I_0^{\rm od}: L^2(M)\to L^2(\pl SM), \quad I_1^{\rm ev}: L^2(M; T^*M)\to L^2(\pl SM)\]
\end{lemm}
\begin{proof}
The operators $I_0^*I_0$ and $I_1^*I_1$ extend to $M^\circ_e$ as smooth pseudo-differential operators of order $-1$ by \cite[Propositions 5.7 and 5.9]{Gu2}, thus one has by using the standard $TT^*$ argument that 
$I_0:L^2(M)\to L^2(\pl_-SM,d\mu_\nu)$ and $I_1:L^2(M;T^*M)\to L^2(\pl_-SM,d\mu_\nu)$ are bounded.
This implies that $(1-\indic_{\{|\cjg v,\nu\cjd|<\eps\}})I_0$ is bounded as a map $L^2(M)\to L^2(\pl_-SM)$
for all $\eps>0$ fixed, and the same for $(1-\indic_{\{|\cjg v,\nu\cjd|<\eps\}})I_1$.
But we also have $\indic_{\{|\cjg v,\nu\cjd|<\eps\}}I_0$ bounded as a map $L^2(M)\to L^2(\pl_-SM)$
if $\eps>0$ is small enough by Lemma \ref{L^2}, and the same for 
$\indic_{\{|\cjg v,\nu\cjd|<\eps\}}I_1$. The proof is complete.
\end{proof}
Notice however that when we will use the adjoints of $I_0$ or $I_1$, this will be always adjoint with respect to the space $L^2(\pl_-SM, d\mu_\nu)$. 

\section{The inversion formula for $I_0$.}

\subsection{The inversion formula}
Let us define the operator 
\begin{equation}\label{DefW}
W:=\frac{1}{2\pi}{\pi_0}_*X_\perp R_+\pi_0^* : C_c^\infty(M_e^\circ)\to C^{-\infty}(M^\circ_e).
\end{equation}
We start by proving the following
\begin{prop}\label{regW}
The operator $W$ defined in \eqref{DefW} has a smooth Schwartz kernel on $M_e^\circ\x M_e^\circ$. If the Gauss curvature $\kappa(x)$ is a negative constant, then $W=0$. 
\end{prop}
\begin{proof} We decompose the operator $R_+$ as $R_+=R_+^1+R_+^2$ with
\[ R^1_+f: = \int_{0}^\eps e^{tX}dt, \quad R_+^2:= e^{\eps X}R_+\]
where $\eps>0$ is smaller than the radius of injectivity of the metric. By \cite{DyGu},  the Schwartz kernel 
of $R_+ $
is  a distribution on $SM^\circ_e\x SM^\circ_e$ with wavefront set 
\begin{equation}\label{WFRpm} 
{\rm WF}( R_+)\subset N^*\Delta(SM^\circ_e\x SM^\circ_e)\cup \Omega_-\cup (E_-^*\x E_+^*)\end{equation} 
where $E_\pm^*$ are defined in \eqref{E_-*},  $N^*\Delta(SM^\circ_e\x SM^\circ_e)$ is the conormal bundle to the diagonal 
$\Delta(SM^\circ_e\x SM^\circ_e)$ of $SM^\circ_e\x SM^\circ_e$ and 
\[\Omega_-:= \{(\varphi_{-t}(y),(d\varphi_{-t}(y)^{-1})^T\xi,y,-\xi)\in T^*(SM^\circ_e\x SM^\circ_e);  \,\, t\geq 0,\,\, \xi(X(y))=0\}.\] 
The wave-front set of the Schwartz kernel of $e^{\eps X}$ is 
\[{\rm WF}(e^{\eps X})\subset \{(\varphi_{-\eps}(y),\eta,y,-d\varphi_{-\eps}(y)^T\eta); \,\, y\in SM^\circ_e, \eta\in 
T^*_{\varphi_{-\eps}(y)}(SM^\circ_e)\setminus\{0\}\}\]
thus by the composition rule of wave-front sets given in \cite[Theorem 8.2.14]{Ho},   
\[\begin{split}
{\rm WF}(e^{\eps X}R_+)\subset & \{(\varphi_{-t}(y),(d\varphi_{-t}(y)^{-1})^{T}\eta,y,-\eta);  \,\, t\geq \eps,\,\, \eta(X(y))=0\}\cup (E_-^*\x E_+^*)\\
& \cup \{(\varphi_{-\eps}(y),\eta ,y,-d\varphi_{-\eps}(y)^T\eta);\,\, (y,\eta)\in T^*(SM^\circ_e)\setminus\{0\}\}.
\end{split}\]
Since a differential operator does not increase wavefront sets, we deduce that  
\begin{equation}\label{WFSet}
\begin{split} 
{\rm WF}(X_\perp e^{\eps X}R_+)\subset  & \{(\varphi_{-t}(y),(d\varphi_{-t}(y)^{-1})^{T}\eta,y,-\eta);  \,\, t\geq \eps,\,\, \eta(X(y))=0\}\cup (E_-^*\x E_+^*)\\
& \cup \{(\varphi_{-\eps}(y),\eta ,y,-d\varphi_{-\eps}(y)^T\eta);\,\, (y,\eta)\in T^*(SM^\circ_e)\setminus\{0\}\}.
\end{split}\end{equation}
The Schwartz kernel of ${\pi_0}_*X_\perp e^{\eps X}R_+\pi_0^*$ is given by 
$(\pi_0\otimes \pi_0)_*K$ if $K$ is the kernel of $X_\perp e^{\eps X}R_+$. Using \eqref{WFSet},� the same exact arguments as in Proposition 5.7 of \cite{Gu2} show that the push-forward $(\pi_0\otimes \pi_0)_*K$ 
is smooth outside the submanifold $\{(x,x')\in M_e^\circ\x M_e^\circ; d_g(x,x')=\eps\}$
if $g$ has no conjugate points. We are reduced to analyzing $W_1:={\pi_0}_*X_\perp R_+^1\pi_0^*$ and this 
is very similar to the case of a simple metric studied in \cite[Prop 5.1]{PeUh}. We denote by $Y(x,v,t)$ 
the Jacobi field along the geodesic $\cup_{t\geq 0}\varphi_t(x,v)$ satisfying $Y(x,v,0)=0$ 
and $\nabla_tY(x,v,0)=Jv$ where $J$ is the rotation of $+\pi/2$ in the fibers of $SM$, and let $Z(x,v,t)$ be the Jacobi field satisfying $Z(x,v,0)=Jv$ and $\nabla_tZ(x,v,0)=0$. They satisfy 
\[ d\pi_0.d\varphi_t(x,v).(X_\perp,0)=Z(x,v,t), \quad d\pi_0.d\varphi_t(x,v).(0,V)=Y(x,v,t)\]
where we used the splitting $TSM=\mc{H}\oplus \mc{V}$. The vector field $Z,Y$ satisfy
\[
Z(x,v,t)=a(x,v,t)Jv(t), \quad Y(x,v,t)=b(x,v,t)Jv(t), \quad \textrm{ with }\]  
\begin{equation}\begin{gathered}
\label{champsjacobi}
\pl_t^2 a(x,v,t)+\kappa(x(t))a(x,v,t)=0,\quad \pl_t^2 b(x,v,t)+\kappa(x(t))b(x,v,t)=0,\\
a(x,v,0)=1, \,\, \pl_ta(x,v,0)=0, \quad b(x,v,t)=0, \,\, \pl_tb(x,v,0)=1
\end{gathered}\end{equation} 
if $\varphi_t(x,v)=(x(t),v(t))$ and $\kappa$ is the Gaussian curvature. We can then write 
\[ \begin{split}
W_1f(x)= & \int_{S_xM}\int_{0}^\eps X_\perp (f(\pi_0(\varphi_t(x,v))))dtdS_{x}(v)\\
&= \int_{S_xM}\int_{0}^\eps df(x(t)).Z(x,v,t)dtdS_{x}(v)\\
&=\int_{S_xM}\int_{0}^\eps \frac{a(x,v,t)}{b(x,v,t)}df(x(t)).Y(x,v,t)dtdS_{x}(v)\\
W_1f(x)&=- \int_{S_xM}\int_{0}^\eps V\Big(\frac{a(x,v,t)}{b(x,v,t)}\Big)f(x(t)) dtdS_{x}(v)
\end{split}\]
and using the change of variable $\exp_x: (t,v)\mapsto y:=\pi_0(\varphi_t(x,v))=x(t)$ satisfying  
$(\exp_x)_*(b(x,v,t)dtdS_x(v))={\rm dvol}_g(y)$, we deduce that the Schwartz kernel of $W_1$ is given by
\begin{equation}\label{kernelW1} 
W_1(x,y)= -\indic_{[0,\eps]}(d_g(x,y)) \Big[\frac{1}{b(x,v,t)}V\Big(\frac{a(x,v,t)}{b(x,v,t)}\Big)\Big]\Big|_{tv=\exp_{x}^{-1}(y)} .
\end{equation}
Clearly this kernel is smooth outside $\{(x,y); d_g(x,y)=\eps \textrm{ or }x=y\}$. We now study the singularity at $x=y$, which is equivalent to $t=0$ in the $(x,t,v)$ coordinates. Using 
\[\pl_t^3 a(x,v,t)+d\kappa_{x(t)}.v(t) a(x,v,t)+\kappa(x(t))\pl_ta(x,v,t)=0\]
and the same for $b$, we make an expansion of $a,b$ at $t=0$
\[ \begin{gathered}
a(x,v,t)=1-\kappa(x)\frac{t^2}{2}-d\kappa_x(v)\frac{t^3}{6}+\mc{O}(t^4), \\
b(x,v,t)=t-\frac{t^3}{6}\kappa(x)+\mc{O}(t^4).
\end{gathered}\]
Thus we get near $t=0$
\begin{equation}\label{kernelint} 
F(x,v,t):=\frac{1}{b(x,v,t)}V\Big(\frac{a(x,v,t)}{b(x,v,t)}\Big)= \mc{O}(t).
\end{equation}
which shows that the kernel extends continuously to the diagonal. Let us show 
that $W_1$ extends smoothly to the diagonal $x=y$. For this, note that 
for each $x_0\in M_e^\circ$ the function $F(x,v,t)$ is smooth in the variable 
$(x,v,t)$ in $U_{x_0}\x [0,\eps]$ where $U_{x_0}$ is a neighborhood of  $S_{x_0}M$ in $SM$.
Then it suffices to show that for each $x$ the Taylor expansion at $t=0$ of $F(x,v,t)$ 
to any order $N\in \nn$ satisfies
\begin{equation}\label{expofF} 
F(x,v,t)=\sum_{k=0}^N F_k(x)(\otimes^k v)t^k+\mc{O}(t^{N+1})
\end{equation}
where $F_k$ are smooth symmetric tensors of order $k$. For this purpose, we have 
$\kappa(x(t))\sim \sum_{k=0}^\infty D^k\kappa(x)(\otimes^k v)\frac{t^k}{k!}$ as $t\to 0$ where $D=\mc{S}\circ \nabla$ is the symmetrized covariant derivative. Then, writing $a(x,v,t)\sim \sum_{k=0}^\infty a_k(x;v)t^k$ as $t\to 0$,
we get for each $k\geq 0$
\[ a_{k+2}(x;v)=-\sum_{i+j=k}\frac{D^i\kappa(x)(\otimes^iv)a_j(x;v)}{i!(k+2)(k+1)}.\]
From this we deduce by a direct induction that $a_k(x;v)=a_k(x)(\otimes^kv)$ is the restriction 
of an element $a_k\in C^\infty(M;\otimes_S^kT^*M)$ 
to $SM$. The same argument applies to $b(x,v,t)/t$. Therefore using that 
$V(c_k(x)(\otimes^kv))=kc_k(x)(Jv,v,\dots,v)$ if $c_k\in C^\infty(M;\otimes_S^kT^*M)$, 
we see that \eqref{expofF} has the desired expansion, showing that $W_1$ extends smoothly to the diagonal. Since $\eps>0$ was arbitrary, we have that $W_1$ has smooth kernel. 

Let us show that $W=0$ if $\kappa<0$ is constant. For ${\rm Re}(\la)>C$ with $C>0$ large, 
we define the operator $W(\la)$ by 
\[W(\la):=\frac{1}{2\pi}{\pi_0}_*X_\perp R_+(\la)\pi_0^*, \quad R_+(\la)f=\int_{0}^\infty e^{-\la t}f\circ \varphi_t \, dt.\]
Due to the exponential damping in $t$, it is easy to see that $W(\la): C_c^\infty(M^\circ_e)\to C^0(M_e^\circ)$ is bounded if $C>0$ is large enough.  In \cite{Gu2} it is shown that there exists $\delta>0$ such that 
$R_+(\la): C_c^\infty(M^\circ_e)\to H^{-s}(SM_e^\circ)$ admits an analytic extension to $\{{\rm Re}(\la)>-\delta\}$ 
if $s>0$ is large enough depending on $\delta$. Thus $W(\la)$ also admits an analytic extension to the same half-plane as a map $C_c^\infty(M^\circ_e)\to H^{-s-1}(SM_e^\circ)$, and we will show it vanishes if $\kappa$ is constant. Since the flow is assumed to have no conjugate points and the trapped set is hyperbolic, there is a uniform lower bound
$|b(x,v,t)|>\eps>0$ for some $\eps>0$, and by using Gronwall lemma there exists $C>0$ such that 
\[|a(x,v,t)|+|b(x,v,t)|\leq Ce^{C|t|} , \quad |V(a(x,v,t))|+|V(b(x,v,t))|\leq Ce^{C|t|}
\]
for all $x,v,t$. 
Using these estimates and reasoning like above,
we have the converging expression for each $f\in C_c^\infty(M_e^\circ)$ 
\begin{equation}\label{Wla}
W(\la)f(x)=- \frac{1}{2\pi}\int_{S_xM}\int_{0}^\infty e^{-\la t} V\Big(\frac{a(x,v,t)}{b(x,v,t)}\Big)f(x(t)) dtdS_{x}(v)
\end{equation}
as a continuous map of $x$ if ${\rm Re}(\la)>2C$. 
Now, by \eqref{champsjacobi}, the functions $a(x,v,t)$ and $b(x,v,t)$ are constant in $(x,v)$ if the curvature $\kappa$ is constant, thus $W(\la)f=0$ for all $f\in C_c^\infty(SM_e^\circ)$. By analyticity in $\la$, it implies that $W=W(0)$ vanishes when $\kappa$ is constant. 
\end{proof}

Using that $X_\perp^*=-X_\perp$ on $C_c^\infty(SM^\circ)$ 
and $(R_+^*u)(x,v)=-(R_+u)(x,-v)$ if $u\in C_c^\infty(SM^\circ)$ 
is odd with respect to $A(x,v)=(x,-v)$, we obtain that the $L^2$-adjoint of $W$ is given by
\[W^*=\frac{1}{2\pi}{\pi_0}_*R_+X_\perp\pi_0^*.\]
Now, we can show that Pestov-Uhlmann inversion formula \cite{PeUh} works also under our assumptions.
\begin{theo}\label{inversion}
 If $(M,g)$ has strictly convex boundary and assumptions \eqref{assumptions} hold, then the following identity holds for each $f\in L^2(M)$ and $h\in H_0^1(M)$
\begin{equation}\label{formulainv0}
f+W^2f= -\frac{1}{4\pi} *dI_1^*\mc{S}_g^{-1}(HI_0^{\rm od}f)|_{\pl_+SM},
\end{equation} 
\begin{equation}\label{formulainv1}
h+(W^*)^2h= -\frac{1}{4\pi}I_0^*\mc{S}^{-1}_g(HI^{\rm ev}_1(*dh))|_{\pl_+SM}
\end{equation}
where $W$ is the smoothing operator defined in Proposition \ref{regW}, $W^*$ is its $L^2$-adjoint, 
$*$ denotes the Hodge star operator on $1-$forms and $\mc{S}_g$ is defined by \eqref{unitary}.
\end{theo}
\begin{proof} We follow  the proof of \cite[Theorem 5.1]{PeUh}. First using \eqref{oddpart} we have as distribution on $SM^\circ$
\begin{equation}\label{XuF} 
0=H\pi_0^*f=H_{\rm ev}XR_+\pi_0^*f =XH_{\rm od}R_+\pi_0^*f+ \pi_0^*Wf
\end{equation}
for any $f\in C_c^\infty (M^\circ)$. Applying $H$ to this identity, we get by \eqref{commutator}
\[ 0=XH^2(R_+\pi_0^*f)_{\rm od}+(X_\perp H_{\rm od}R_+\pi_0^*f)_0=-X(R_+\pi_0^*f)_{\rm od}+(X_\perp H_{\rm od}R_+\pi_0^*f)_0\]
but since $X(R_+\pi_0^*f)_{\rm od}=XR_+\pi_0^*f=-f$ as 
$X$ maps odd functions to even functions (and conversely), we obtain
\begin{equation}\label{fXperp}
f=-(X_\perp H_{\rm od}R_+\pi_0^*f)_0.
\end{equation}
The unique $L^1$-solution of $Xu=-F$, $u|_{\pl_+SM}=\omega$ with $F\in C^\infty(SM)$ and $\omega\in L^2(\pl_+SM, d\mu_\nu)$ is given by $u=R_+F+I^*\mc{S}_g^{-1}\omega$. But $\mc{S}_g^{-1}(HI_0^{\rm od}f)|_{\pl_+SM}\in L^2(\pl_-SM,d\mu_\nu)$ by Lemma \ref{I_0L^2} and \eqref{unitary}, and $Wf\in C^\infty(M)$ by Proposition \ref{regW}, we apply this to \eqref{XuF}� and get 
\[ H_{\rm od}R_+\pi_0^*f=R_+\pi_0^*Wf+\demi I^*\mc{S}_g^{-1}(HI_0^{\rm od}f)|_{\pl_+SM}.\]
Notice that we have used that $(x,v)\mapsto R_+f(x,-v)$ vanishes on $\pl_+SM$ since $R_+f=0$ on $\pl_-SM$.
Applying $\frac{1}{2\pi}{\pi_0}_*X_\perp$ we get 
\[ (X_\perp H_{\rm od}R_+\pi_0^*f)_0=W^2f+\tfrac{1}{4\pi} {\pi_0}_*X_\perp I^*\mc{S}_g^{-1}(HI_0^{\rm od}f)|_{\pl_+SM} \]
which gives \eqref{formulainv0} by using \eqref{fXperp}, the identity ${\pi_1}_*I^*=I_1^*$
and  ${\pi_0}_*X_\perp=*d{\pi_1}_*$. The extension to $f\in L^p(M)$ for $p\geq 2$ is obtained by 
density and boundedness of each of the operators in the formula on the correct spaces. Notice that a priori
$*dI_1^*\mc{S}_g^{-1}HI_0^{\rm od}f\in H^{-1}(M^\circ)$ if $f\in L^2(M)$ but \eqref{formulainv0} actually shows it is in $L^2(M)$. 

Next we prove the inversion formula for co-exact $1$-forms. Let $h\in C_c^\infty(M^\circ)$, then since 
$X_\perp \pi_0^*h$ is odd, we get
\[ X(R_+X_\perp \pi_0^*h)_{\rm ev} = -X_\perp \pi_0^*h \]
and applying $H$ with \eqref{oddpart}, this gives
\[ XH(R_+X_\perp \pi_0^*h)_{\rm ev}=-HX_\perp \pi_0^*h-X_\perp \pi_0^*W^*h= X\pi_0^*h- X_\perp \pi_0^*W^*h
\]   
by using that $-H=V$ on $\Omega_1\oplus \Omega_{-1}$ and $VX_\perp=X$ on $\Omega_0$. Here 
$H(R_+X_\perp \pi_0^*h)_{\rm ev}\in L^1(SM)\cap C^0(SM\setminus \Gamma_-)$ and its 
restriction to $\pl SM$ is $\demi HI^{\rm ev}_1(*dh)\in L^2(\pl SM)$ by using Lemma \ref{I_0I_1}, thus 
\[ \begin{split}
H(R_+X_\perp \pi_0^*h)_{\rm ev}= & -R_+X\pi_0^*h+R_+X_\perp \pi_0^*W^*h+\demi I^*\mc{S}_{g}^{-1}(HI^{\rm ev}_1(*dh))|_{\pl_+SM}\\
 = &\pi_0^*h+R_+X_\perp \pi_0^*W^*h+\demi I^*\mc{S}^{-1}_g(HI^{\rm ev}_1(*dh))|_{\pl_+SM}
 \end{split}\]
 We apply $\frac{1}{2\pi}{\pi_0}_*$ and obtain 
 \[ h+(W^*)^2h= -\frac{1}{4\pi}I_0^*\mc{S}^{-1}_g(HI^{\rm ev}_1(*dh))|_{\pl_+SM}.\]
 By density the same identity holds for $h\in H_0^1(M)$.
\end{proof}
We notice that a divergence-free $1$-form $u$ with Sobolev regularity $H^2(M;T^*M)$ can be decomposed under the form $u=*dh +w$ where 
$h\in H_0^1(M)\cap H^3(M)$ and $w\in H^2(M;T^*M)$ satisfies $dw=0$ and $d*w=0$. The formula \eqref{formulainv1}
allows to recover a divergence-free $1$ form $*dh$ from its X-ray transform, and the X-ray transform of $I_1(w)(x,v)$ of $w$ at $(x,v)\in \pl_-SM\setminus \Gamma_-$ 
is given purely in terms of  the integral of $w$ on a curve 
in $\pl M$ with endpoints $x$ and $\pi_0(S_g(x,v))$ and the homology class of 
the geodesic $\pi_0(\varphi_t(x,v))$ in $M$, which is a known data.

\subsection{Estimates of the norm of the error term in negative curvature}
We now show that in pinched negative curvature the operator $W$ has small norm on $L^p(M)$, 
allowing to invert ${\rm Id}-W^2$ by Neumann series. 

First, we have by \cite[Section 5]{Gu2} that $\demi I_0^*I_0={\pi_0}_*R_+\pi_0^*$ and since 
\begin{equation}\label{R+L1}
R_+f(x,v)=\int_{0}^{\ell_g(x,v)}\pi_0^*f(\varphi_t(x,v))dt  \in L^1(SM)
\end{equation}
by \cite[Proposition 4.2]{Gu2} if $f\in L^p(M)$ with $p>1$, we see that the integral
\[I_0^*I_0f(x)=2 \int_{S_xM}\int_{0}^{\ell_g(x,v)}\pi_0^*f(\varphi_t(x,v))dtdS_x(v)\]
is defined for almost every $x\in M$ if $f\in L^p(M)$ with $p>1$, by Fubini theorem. 
In \cite[Proposition 5.7]{Gu2}, it is shown that 
$I_0^*I_0$ is the restriction to $M$ of an elliptic pseudo-differential operator of order $-1$ on the extended manifold $M^\circ_e$, thus it maps $L^p(M)$ to $L^p(M)$ for each $p\in [1,\infty]$.

\begin{prop}\label{normW}
Let $(M,g)$ be a surface with strictly convex boundary and assume that the Gauss curvature $\kappa\in C^\infty(M)$ of $g$ is negative, with $\kappa_0=\min_{x\in M}|\kappa(x)|$. For each $p\in [1,\infty]$, 
the $L^p\to L^p$ norm of $W$ is bounded by 
\[||W||_{L^p\to L^p}\leq \frac{|d\kappa|_{L^\infty}}{3\kappa_0}||I_0^*I_0||_{L^p\to L^p}.\]
\end{prop}
\begin{proof} 
Let $a(x,v,t)$ and $b(x,v,t)$ the functions defined by \eqref{champsjacobi}. If we can prove that 
\[ \Big|V\Big(\frac{a(x,v,t)}{b(x,v,t)}\Big)\Big|\leq C_0\]
uniformly in $x,v,t$, then the integral in 
\eqref{Wla} is convergent for almost every $x$ when $\la=0$, for each $f\in L^2$. 
We get in this case that for each $f\in L^p(M)$ with $p\in [2,\infty]$
\[ ||Wf||_{L^p(M)}\leq C_0\Big(\int_{M}\Big|\int_{0}^{\ell(x,v)}|\pi_0^*f(\varphi_t(x,v))|dt dS_x(v)\Big|^pdx\Big)^{1/p}
= \frac{C_0}{2} ||I_0^*I_0(|f|)||_{L^p(M)}.
\]
Clearly if $\kappa\leq 0$, we have $a(x,v,t)\geq 1$ and $b(x,v,t)\geq 0$, and both are increasing in $t$.
First, by Wronskian constancy, we get for all $t>0$ and $(x,v)\in SM$
\[ \pl_t\Big(\frac{a(x,v,t)}{b(x,v,t)}\Big)=\frac{-1}{b^2(x,v,t)}\]
To simplify notations, we will often drop the $(x,v)$ dependence below.
The function $r(s):=\dot{b}(s)/b(s)$ satisfies the Riccati equation $\dot{r}(s)+r^2(s)+\kappa(x(s))=0$
with $r(s)>0$ and $r(s)\to +\infty$ as $s\to 0^+$. We claim that $r(s)\geq \sqrt{\kappa_0}$ for all $s>0$: 
indeed if it were not the case, there is $\eps>0$ small and $s>0$ so that $r^2(s)=\kappa_0-\eps$. 
Let $s_0>0$ be the first time so that this happens, then $\dot{r}(s_0)=-\kappa(x(s_0))-r^2(s_0)\geq 
\eps$ and thus $r(s)<\kappa_0-\eps$ for all $s<s_0$ close enough to $s_0$, leading to a contradiction.
We then get for $s>t>0$ 
\begin{equation}\label{btbs}  
b(x,v,s)\geq b(x,v,t)e^{\sqrt{\kappa_0}(s-t)}
\end{equation}
Notice that by Sturm comparison theorem,  we also have for $t\geq 0$
\begin{equation}\label{Rauch}
b(x,v,t)\geq \frac{1}{\sqrt{\kappa_0}}\sinh(\sqrt{\kappa_0}t).
\end{equation}
Then, when $(x,v)\in \Gamma_-$, $c(x,v):=\lim_{t\to +\infty}\frac{a(x,v,t)}{b(x,v,t)}$ exists and is continuous in $(x,v)$, moreover for $t>0$ we deduce from \eqref{btbs}
\[ 0\leq \frac{a(x,v,t)}{b(x,v,t)}-c(x,v)=\int_{t}^{\infty}\frac{1}{b^2(x,v,s)}ds\leq \frac{1}{2\sqrt{\kappa_0}b^2(x,v,t)}.\]
In fact we have more generally the same estimate for all $(x,v)\in SM^\circ$: for $t\in(0,\ell(x,v)]$
\begin{equation}\label{estimeea/b} 
0\leq \frac{a(x,v,t)}{b(x,v,t)}-c(x,v)\leq \int_{t}^{\ell(x,v)}\frac{1}{b^2(x,v,s)}ds
\end{equation}
where we have denoted $c(x,v):=\frac{a(x,v,\ell(x,v))}{b(x,v,\ell(x,v))}$.
Let us set $B(x,v,t)=Vb(x,v,t)$ and $A(x,v,t)=Va(x,v,t)$.
 We have 
\[ \ddot{A}(t)+\kappa(x(t))A(t)=-b(t)a(t)\kappa_{\perp}(t) , \quad \ddot{B}(t)+\kappa(x(t))B(t)=-b^2(t)\kappa_{\perp}(t)\]
where $\kappa_\perp(t)=d\kappa(x(t)).Jv(t)$ and $\varphi_t(x,v)=(x(t),v(t))$. We notice that 
$\kappa_\perp(t)=(X_\perp \pi_0^*\kappa)(\varphi_t(x,v))$.
 By Duhamel formula, we obtain 
\[ \begin{gathered}
A(t)=-\int_{0}^t (a(s)b(t)-a(t)b(s))a(s)b(s)\kappa_\perp(s)ds, \\
B(t)=-\int_{0}^t (a(s)b(t)-a(t)b(s))b^2(s)\kappa_\perp(s)ds
\end{gathered}
\]
which gives 
\begin{equation}\label{Va/b} 
V\Big(\frac{a}{b}\Big)(t)=-\frac{1}{b^2(t)}\int_{0}^t (a(s)b(t)-a(t)b(s))^2b(s)\kappa_\perp(s)ds.
\end{equation}
We want to give a bound on this quantity. 
We see that for $\ell(x,v)>t>s>0$, 
\[ (a(s)b(t)-a(t)b(s))^2\leq b(t)^2b(s)^2\Big(\int_{s}^{\ell(x,v)}\frac{1}{b(u)^2}du\Big)^2.\]
Combining with  \eqref{Va/b}, we deduce the uniform pointwise estimates
\begin{equation}\label{2estimees}
    \begin{gathered} 
	\Big|V\Big(\frac{a(t)}{b(t)}\Big)\Big|\leq 
	|d\kappa|_{L^\infty}\int_0^tb(s)^3\Big(\int_s^{\ell(x,v)} \frac{1}{b(u)^2}du\Big)^2ds. 
    \end{gathered}
\end{equation}
Using \eqref{Rauch} and \eqref{btbs}, we obtain  
\[\Big|V\Big(\frac{a(t)}{b(t)}\Big)\Big|\leq 
	\frac{|d\kappa|_{L^\infty}}{\kappa_0} 
	\int_0^\infty e^{3s}\Big(\int_{s}^\infty \frac{e^{-\frac{3}{2}u}}{\sqrt{\sinh u}}du
	\Big)^2ds
	  \]
and the double integral is a finite constant which, after some calculation turns out to be equal to $2/3$.
\end{proof}

\subsection{The norm of $I_0^*I_0$}\label{normofPi0}
We will  give an estimate in negative curvature of the norm of the operator 
\begin{equation}\label{defofPi}
\Pi_0:=I_0^*I_0: L^p(M)\to L^p(M)
\end{equation} 
by considering the geometry in the universal cover.
We let $\til{M}$ be the universal cover of $M$, equipped with the pull-back metric $\tilde{g}$ of $g$.  
Since $g$ has negative curvature, the manifold $\til{M}$ is a non-compact simply connected manifold with boundary and the fundamental group $\Gamma:=\pi_1(M)$ of $M$ acts properly discontinuously on $\til{M}$ by isometries with respect to $\til{g}$. Denote by $\pi_{\Gamma}:\til{M}\to M$ the covering map, and by abuse of notation we also call $\pi_\Gamma:S\til{M}\to SM$ the covering map where $S\til{M}$ is the unit tangent bundle of $(\til{M},\til{g})$. Fix a fundamental domain 
$F\subset \til{M}$ for the action of $\Gamma$, then $F$ is compact since $M$ is compact. 
If $f\in L^p(M)$ with $p\in(1,\infty)$, let $\tilde{f}$ be its lift to $\til{M}$ (which is $\Gamma$-automorphic), then we have 
for almost every $\til{x}\in F$ and $x:=\pi_{\Gamma}(\til{x})$ 
\begin{equation}\label{Pi0} 
\Pi_0f(x)= 2\int_{S_x\til{M}}\int_{0}^{\tilde{\ell}_{+}(\til{x},\til{v})} (\pi_0^*\tilde{f})(\til{\varphi}_{t}(\tilde{x},\tilde{v}))dt dS_{\tilde{x}}(\til{v})
\end{equation} 
where $\tilde{\varphi}_t$ is the geodesic flow lifted to $S\til{M}$ and $\tilde{\ell}_{+}=\pi_{\Gamma}^*\ell_+$. 
Define the operator $\til{\Pi}_0$ on $\til{M}$ as follows: for $u\in L^p(\til{M})$ with compact support, let 
\[\til{\Pi}_0u(\tilde{x}):= 2\int_{S_x\til{M}}\int_{0}^{\tilde{\ell}_{+}(\til{x},\til{v})} (\pi_0^*u)(\til{\varphi}_{t}(\tilde{x},\tilde{v}))dt dS_{\tilde{x}}(\til{v})\]
which is in $L^1_{\rm loc}(\til{M})$ by \eqref{R+L1}. 
With the same analysis as in Proposition \ref{DefW}, we find the integral kernel of $\til{\Pi}_0:$ 
\[\begin{gathered} \til{\Pi}_0u(\tilde{x})=\int_{\til{M}}\mc{K}(\tilde{x},\tilde{x}') 
u(\tilde{x}')
{\rm dvol}_{\tilde{g}}(\tilde{x}') , \\ 
\textrm{ with }\mc{K}(\tilde{x},\tilde{x}') :=\Big[\frac{2}{\tilde{b}(\tilde{x},\tilde{v},t)}\Big]\Big|_{t\tilde{v}=\til{\exp}_{\tilde{x}}^{-1}(\tilde{x}')} \in C^\infty(\til{M}\x\til{M}\setminus {\rm diag})
\end{gathered}
\]
if $\til{\exp}$ is the exponential map for $\til{M}$ and $\tilde{b}$ solves the same Jacobi equation as \eqref{champsjacobi} but for $(\til{M},\tilde{g})$. We clearly have $\tilde{b}(\tilde{x},\tilde{v},t)=b(x,v,t)$ if $\pi_\Gamma(\tilde{x},\tilde{v})=(x,v)$ and the integral kernel $\mc{K}$ of $\til{\Pi}_0$ is positive, symmetric and satisfies 
$\mc{K}(\gamma \tilde{x},\gamma \tilde{x}')=\mc{K}(\tilde{x},\tilde{x}')$ for all $\gamma\in \Gamma$. From \eqref{Pi0}, we can write for almost every  $x=\pi_\Gamma(\tilde{x})$
\[ \Pi_0f(x)= \int_{\til{M}}\mc{K}(\tilde{x},\tilde{x}')\tilde{f}(\tilde{x}'){\rm dvol}_{\tilde{g}}(\tilde{x}').\]
Let $f_0\in L^p(\til{M})$ be supported in $F$ and such that $f_0(\tilde{x})=f(x)$ 
if $\tilde{x}\in F$, then $\tilde{f}(\tilde{x})=\sum_{\gamma\in \Gamma}f_0(\gamma \tilde{x})$ and we get for almost every $\tilde{x}\in F$ and $\pi_\Gamma(\tilde{x})=x$
\[
\Pi_0f(x)= \sum_{\gamma\in\Gamma}\int_{\til{M}}\mc{K}(\tilde{x},\tilde{x}')f_0(\gamma\tilde{x}'){\rm dvol}_{\tilde{g}}(\tilde{x}')
 = \sum_{\gamma\in\Gamma}\int_{F}\mc{K}(\gamma\tilde{x},\tilde{x}')f_0(\tilde{x}'){\rm dvol}_{\tilde{g}}(\tilde{x}')
\]
by using that $\gamma$ is an isometry of $\tilde{g}$ and the invariance of $\mc{K}$ by $\Gamma$ mentioned
above. We can then view $\Pi_0$ as an operator acting on $L^p(F)$ (for $p\in(1,\infty)$)
with integral kernel 
\[ \mc{K}_M(\tilde{x},\tilde{x}'):=\sum_{\gamma\in \Gamma}\mc{K}(\gamma\tilde{x},\tilde{x}')=
\sum_{\gamma\in \Gamma}\mc{K}(\tilde{x},\gamma\tilde{x}')
.\]
The sum makes sense as an $L^1(F\x F)$ function since $\mc{K}>0$ and $\Pi_0f\in L^1(M)$
for all $f\in L^\infty(M)$. In fact, we also know that $\Pi_0$ maps $L^\infty(M)$ to 
$L^\infty(M)$ (since it is the restriction of a pseudo-differential operator of order $-1$ on $M_e^\circ$)
thus, since $\mc{K}_M>0$, 
\[\sup_{\tilde{x}\in F}\int_{F}\mc{K}_M(\tilde{x},\tilde{x}'){\rm dvol}_{\tilde{g}}(\tilde{x}')<\infty.\]
To estimate its $L^p\to L^p$ norm for $p\in(1,\infty)$, we can use Schur's lemma and since $\mc{K}_M$ is symmetric, this gives directly the following
\begin{lemm}\label{normofPi}
With the notation above, the operator $\Pi_0$ has norm bounded by 
\[||\Pi_0||_{L^p(M)\to L^p(M)}\leq 
\sup_{\tilde{x}\in F}\int_{F}\mc{K}_M(\tilde{x},\tilde{x}'){\rm dvol}_{\tilde{g}}(\tilde{x}')
=\sup_{\tilde{x}\in F}\int_{\til{M}}\mc{K}(\tilde{x},\tilde{x}'){\rm dvol}_{\tilde{g}}(\tilde{x}').\]
\end{lemm}

From now on, we shall work in a neighborhood of a constant curvature metric $g_0$ and estimate 
the norm $\Pi_0$ in that neighborhood. Let us denote by $\hh^2$ the hyperbolic space with curvature $-1$, realized as the unit disk $\mathbb{D}^2\subset \cc$ with metric 
$g_{\hh^2}=\frac{4|dz|^2}{(1-|z|^2)^2}$. 
Let $(M,g_0)$ have constant curvature $\kappa(x)=-\kappa_0$ for some $\kappa_0<0$, 
we can extend $M$ to a non-compact complete manifold $(M_c,g_c)$ with constant curvature.
Indeed, this can be done as follows. The connected components  $(S_i)_{i=1,\dots,N}$ of the boundary $\pl M$ are circles. 
Let $\tau_i>0$ be the distance from $S_i$ to the unique geodesic $\gamma_i$ 
homotopic to $S_i$ in $M$, these geodesics form the boundary of the convex core $C(M)$ of $M$. If $c_i$ denotes the length of $\gamma_i$,
then in normal geodesic coordinates $(\tau,\alpha)\in [0,\tau_i]\x (\rr/c_i\zz)_{\alpha}$ from $\gamma_i$ in the collar between $S_i$ and $\gamma_i$, the metric  is given by
 \begin{equation}\label{gfunnel} 
 g_0= d\tau^2+ \cosh(\sqrt{\kappa_0}\tau)d\alpha^2.
 \end{equation}
Then it suffices to replace the region between $S_i$ and $\gamma_i$ by the full hyperbolic half-cylinder $[0,\infty)_\tau\x (\rr/c_i\zz)_{\alpha}$ with the metric \eqref{gfunnel} and we denote by $(M_c,g_c)$ the obtained complete surface. 
It is realized as a quotient $\Gamma\backslash \hh_{\kappa_0}^2$
by a convex co-compact group of isometries $\Gamma\subset {\rm PSL}_2(\rr)$ 
of the hyperbolic space $\hh_{\kappa_0}^2$ with curvature $-\kappa_0$; this space $\hh^2_{\kappa_0}$ is just the unit disk $\mathbb{D}^2\subset \cc$ with metric 
$\kappa_0^{-1}g_{\hh^2}$. Notice that $\Gamma$ is also convex a co-compact group of isometries of 
$(\hh^{2},g_{\hh^2})$.
The exponent of convergence of $\Gamma$ is defined to be 
\begin{equation}\label{deltaGamma} 
\delta_{\Gamma}:=\inf \{s \in(0,1) ; \sum_{\gamma\in \Gamma} e^{-sd_{\hh^2}(x,\gamma x)}<\infty\}
\end{equation}
where $d_{\hh^2}(\cdot,\cdot)$ denotes the Riemannian distance of $(\hh^{2},g_{\hh^2})$ 
and $x$ is any point fixed in $\hh^{2}$. 
Let $\Lambda_\Gamma\subset \mathbb{S}^1$ be the limit set of the group $\Gamma$ (ie. the set of accumulation points in $\bbar{\mathbb{D}^2}$ of the orbits $\Gamma.x$ in $\mathbb{D}^2$).
By a result of Patterson \cite{Pat} and Sullivan \cite{Su}, 
\[\delta_{\Gamma}=\dim_{\rm Haus}(\Lambda_\Gamma)=\demi (\dim_{\rm Haus}(K)-1)\]
where $K\subset SM\subset SM_c$ is the trapped set of the geodesic flow. The universal cover 
$\til{M}$ of $M$ is a subset of the open disk $\mathbb{D}^2$ with infinitely many boundary components which are pieces of circles intersecting $\mathbb{S}^1=\pl\mathbb{D}^2$ at their endpoints. The solution of  \eqref{champsjacobi} is $b(x,v,t)=\frac{1}{\sqrt{\kappa_0}}\sinh(\sqrt{\kappa_0}t)$ and the integral kernel 
 of $\til{\Pi}_0$ and $\Pi_0$ are thus
\begin{equation}\label{mcK} 
\mc{K}(\tilde{x},\tilde{x}')=\frac{2\sqrt{\kappa_0}}{\sinh(d_{\hh^2}(\tilde{x},\tilde{x}'))}, 
\quad \mc{K}_M(x,x')=\sum_{\gamma\in \Gamma}
\frac{2\sqrt{\kappa_0}}{\sinh(d_{\hh^2}(\gamma\tilde{x},\tilde{x}'))}\end{equation}
if $\pi_{\Gamma}(\tilde{x})=x$, $\pi_{\Gamma}(\tilde{x'})=x'$ and $\tilde{x},\tilde{x}'\in F$ with $F\subset \hh^2$ 
a fundamental domain. Notice that the expression \eqref{mcK} for $\mc{K}_M$ extends to $M_c\x M_c$ and this defines
an operator $\Pi_0^{M_c}$ so that $\Pi_0f=(\Pi_{0}^{M_c}f)|_{M}$ if $f\in L^2(M)$.

We prove the following 
\begin{theo}\label{normL^2Pi0}
Let $(M,g_0)$ be a manifold with strictly convex boundary and constant negative curvature $-\kappa_0$, let 
$\Gamma\subset {\rm PSL}_2(\rr)$ be its fundamental group and $\delta_{\Gamma}\in [0,1)$ the exponent of convergence of $\Gamma$ as defined in \eqref{deltaGamma}. 
Let $\la_1,\la_2 \in (0,1]$ which satisfy the bound 
$1\geq \la_1\la_2>\max (\delta_{\Gamma},\demi)$ and set $1-\la:=\la_1\la_2$, then 
for all metric $g$ on $M$ with strictly convex boundary and Gauss curvature $\kappa(x)$ satisfying
\[
\la_1^{2}g_0\leq g\leq \la_1^{-2} g_0, \quad  \kappa(x) \leq -\la_2^2\kappa_0
\]
the operator $\Pi_0=I_0^*I_0$ of $g$ has norm bounded by 
\[\begin{gathered}
 ||\Pi_0||_{L^2(M,g)\to L^2(M,g)}\leq \frac{4\la_2}{\sqrt{\kappa_0}\la_1^4}\Big(\frac{\Gamma(\frac{1-2\la}{4})\Gamma(\frac{1+2\la}{4})}{\Gamma(\frac{1}{4})\Gamma(\frac{3}{4})}\Big)^2
\quad \textrm{ if } \delta_{\Gamma}\leq 1/2\\
||\Pi_0||_{L^2(M,g)\to L^2(M,g)}\leq \frac{4\la_2}{\sqrt{\kappa_0}\la_1^4}\frac{\Gamma(\frac{\delta_{\Gamma}-\la}{2})\Gamma(\frac{\delta_{\Gamma}+\la}{2})\Gamma(\frac{1-\la-\delta_{\Gamma}}{2})\Gamma(\frac{1+\la-\delta_{\Gamma}}{2})}
{\Gamma(\frac{1-\delta_{\Gamma}}{2})\Gamma(\frac{\delta_\Gamma}{2})\Gamma(1-\frac{\delta_{\Gamma}}{2})\Gamma(1+\frac{\delta_{\Gamma}}{2})}  \quad \textrm{ if }\delta_{\Gamma}>1/2.
\end{gathered}\]
\end{theo} 
\begin{proof} Let $\la\in[0,1/2)$ so that $1-\la=\la_2\la_1$. 
First, by the comparison result \eqref{Rauch}, we have the pointwise bound on the kernel 
$\mc{K}$ of $\til{\Pi}_0$ in the universal cover $\til{M}\subset \mathbb{D}^2$
\[\begin{split} 
\mc{K}(x,x')\leq &\frac{\sqrt{\kappa_0}\la_2}{\sinh(\la_2\sqrt{\kappa_0}d_{g}(x,x'))}\leq 
\frac{\la_2\sqrt{\kappa_0}}{\sinh(\la_1\la_2\sqrt{\kappa_0}d_{g_0}(x,x'))}
\leq 
\frac{\la_2\sqrt{\kappa_0}}{\sinh((1-\la)d_{\hh^2}(x,x'))}
\end{split}\] 
where we used $d_g(x,x')\geq \la_1d_{g_0}(x,x')$ and $d_{\hh^2}(x,x')=\sqrt{\kappa_0}d_{g_0}(x,x')$ on $\til{M}$.
Using the lower bound $\sinh((1-\la)t)\cosh(\la t)\geq \demi \sinh(t)$ for $t\geq 0$, we get 
\[\mc{K}(x,x')\leq \frac{2\la_2\sqrt{\kappa_0}\cosh(\la d_{\hh^2}(x,x'))}{\sinh(d_{\hh^2}(x,x'))}\]
This implies that the kernel $\mc{K}_M$ of $\Pi_0$ satisfies 
\begin{equation}\label{upperboundmcK} 
\mc{K}_M(x,x') \leq \la_2\sqrt{\kappa_0}\sum_{\gamma\in \Gamma}
\frac{2\cosh(\la d_{\hh^2}(x,\gamma x'))}{\sinh(d_{\hh^2}(x,\gamma x'))}
\end{equation}
provided that the sum converges, which is the case if $\la<(1-\delta_{\Gamma})$. 
Let $g_1:=\kappa_0 g_0$ be the metric with Gauss curvature $-1$ on $M$ and let
$\Pi_0^{\la}$ be the operator on $M$ whose integral kernel with respect to the volume form ${\rm dvol}_{g_1}$ 
is $\sum_{\gamma\in \Gamma}
\frac{2\cosh(\la d_{\hh^2}(x,\gamma x'))}{\sinh(d_{\hh^2}(x,\gamma x'))}$. Let $M_c=\Gamma\backslash \hh^2$ 
equipped with the hyperbolic metric $g_c$ (with curvature $-1$) and consider the operator 
$\Pi_0^{M_c,\la}$ on $M_c$ defined by \eqref{Pi0lambda}. As explained above, 
the manifold $M_c$ can be viewed as a complete extension of the manifold with boundary $(M,g_1)$, and we have for each $f\in L^2(M,g_1)$ extended by $0$ on $M_c$
\[ (\Pi_0^{M_c,\la}f)|_{M}=\Pi_0^\la f.\] 
The operator  norm of $\Pi_0^{M_c,\la}$ on $L^2(M_c,g_c)$ is computed in Lemma \ref{Pi0hyp} and we denote this constant by $C(\la,\delta_{\Gamma}):=||\Pi_0^{M_c,\la}||_{L^2(M_c,g_c)\to L^2(M_c,g_c)}$. We then get 
\begin{equation}\label{compare}
||\Pi_0^\la f||_{L^2(M,g_1)}\leq C(\la,\delta_{\Gamma})||f||_{L^2(M,g_1)}.
\end{equation} 
Let $f\in L^2(M,g)$, we use \eqref{upperboundmcK} together with the fact �that the integral kernels of $\Pi_0$ and $\Pi_0^{\la}$ operator are positive and that 
\[ \frac{\la_1^2}{\kappa_0}\,{\rm dvol}_{g_1}\leq {\rm dvol}_g\leq \frac{\la_1^{-2}}{\kappa_0}
 \,{\rm dvol}_{g_1},\]
this gives the bound
\[\begin{split}
 \int_M |\Pi_0f(x)|^2{\rm dvol}_g(x)\leq & \la_2^2\la_1^{-6}\kappa_0^{-2}\int_M (\Pi_0^{\la}(|f|)(x))^2{\rm dvol}_{g_1}(x)\\
& \leq ||\Pi_0^{\la}||^2_{L^2(M,g_1)\to L^2(M,g_1)}\la_2^2\la_1^{-6}\kappa_0^{-2}||f||^2_{L^2(M,g_1)}\\
& \leq ||\Pi_0^{\la}||^2_{L^2(M,g_1)\to L^2(M,g_1)}\la_2^2\la_1^{-8}\kappa_0^{-1}||f||^2_{L^2(M,g)}.
\end{split}\]
The result then follows from \eqref{compare} and the value for $C(\la,\delta_\Gamma)$ given by Lemma \ref{Pi0hyp}.
\end{proof}
Combining Proposition \ref{normW} with Theorem \ref{normL^2Pi0}, we obtain Theorem \ref{Theo2}. 

%%%%%%%%%%%%%%%%%%%%%%%%%%%%%%%%%%%%%%%%%%%%%%%%%%%%%%%%%%%%%%%%%% NUMERICS

\section{Numerical experiments}
We now illustrate with numerical examples the reconstruction formula \eqref{formulainv0} proposed in Theorem \ref{inversion}, i.e., we will reconstruct functions from knowledge of their ray transform, where the underlying surface with boundary is of one of the following two models: 
\begin{itemize}
    \item[$(i)$] A quotient of the Poincar\'e disk $\dd^2$ by a Schottky group $\Gamma$. 
    \item[$(ii)$] The cylinder $\rr/2\zz \times (-1,1)$ with a metric of circular symmetry such that the trapped set is     hyperbolic
\end{itemize}

In example $(i)$, the constancy of the curvature makes the operator $W$ vanish identically, so that the inversion is expected to be one-shot. In example $(ii)$, we will implement a partial Neumann series to invert the operator ${\rm Id} + W^2$. Much of the ideas and implementation draw from the second author's previous implementation \cite{Mo} in the case of non-trapping surfaces. The numerical domain can be viewed as a domain in $\mathbb{R}^2$ with an isotropic metric $g = e^{2\phi} (dx^2 + dy^2)$ (except for Experiment 5, where we use a metric of the form $g = h^2(y) dx^2 + dy^2$), and with certain identifications producing nontrivial topology and trapped geodesics.

The underlying grid is cartesian (subset of a $N \times N$ uniform discretization of $[-1,1]^2$ with $N=300$ here), which simplifies the computation of the operator $*d$. More specifically, in isothermal coordinates $(x,y,\theta)$ with an isotropic metric of the form $g = e^{2\phi} (dx^2 + dy^2)$, the operator $* d I_1^\star\mc{S}_g^{-1}$ appearing in \eqref{formulainv0} takes the expression
\begin{align*}
    * d I_1^\star \mc{S}_g^{-1} w (x,y) = e^{-2\phi(x,y)} \binom{\partial_{x}}{\partial_{y}} \cdot \left( e^{\phi} \int_{\sph^1} \binom{-\sin\theta}{\cos\theta} w (\varphi_{\ell_+(x,y,\theta)}(x,y,\theta))\ d\theta \right),
\end{align*}
where, despite the fact that this formula might not make sense for those $(x,y,\theta)$ in the trapped set, the hyperbolicity of the flow is such that, when discretizing the integral over $\sph^1$, numerical imperfections will make every geodesic exit the domain. The operators $\partial_{x}$, $\partial_{y}$ can then conveniently computed by finite differences. On the other hand, this discretization implies that every pixel will carry a different volume, and that accuracy of reconstructions and resolution may be visually position-dependent. To the authors' knowledge, litterature on numerical analysis of ray transforms on manifolds (covering, e.g., error estimates, appropriate sampling, regularization), is very scarce outside the well-understood Euclidean case \cite{N}, an exception being on geodesic and horocyclic transforms on the hyperbolic plane in \cite{Fr}, whose homogeneity allows for harmonic analysis.

The computation of geodesics is done in two ways depending on the purpose: 
\begin{itemize}
    \item When computing the ray transform of a given function, we solve an ODE and integrate the function along the way. The method used in that case is Heun's scheme.  
    \item In the case of quotients of $\dd$, where geodesics can be given by explicit formulas, we use these formulas to compute geodesic endpoints in a fast way during the backprojection step. 
\end{itemize}
Computations are performed with {\tt Matlab}, using GPU in order to speed up calculations. As in \cite{Mo}, the main bottleneck is the backprojection step, which requires computing ${\mathcal O}(N^2\times N_{\theta})$ geodesic endpoints, with $N$ the gridsize and $N_{\theta}$ the number of directions used to compute integral over $\mathbb{S}^1$. This process requires computing whole geodesics, more of whom take longer to exit the domain when trapping increases (this increase here is quantified by the fractal dimension of the limit set of the corresponding Schottky group). Accordingly, computations run within minutes on a personal computer with GPU card for the case of cylinders, hours for the case of Schottky groups with two generators.  

\subsection{Quotients of $\dd^2$ by Schottky groups with one generator} Here and below, we denote $\gamma_{\rmx,v}$ the unique geodesic passing through $\rmx = (x,y)$ with direction $v$. 

Fix $x\in (-1,0)$, let $a = \frac{2x}{x^2+1}$ and define $T_a:\dd^2\to \dd^2$ by $T_a(z) = \frac{z-a}{1-\overline{a}z}$ the unique hyperbolic translation mapping $x$ to $-x$ and define the surface $M(x) = \dd^2/\langle T_a \rangle$. A fundamental domain for this surface is delimited by the geodesics $\gamma_{(x,0),\bfe_y}$ and $\gamma_{(-x,0),\bfe_y}$. This surface has no boundary, so we cut one up using an arc circle which intersects the geodesic $\gamma_{(x,0),\bfe_y}$ at a right angle, which ensures 
that the boundary is smooth. This boundary has two connected components. The forward trapped set consists of, at every point, precisely two directions: the ones whose flow-out lands into either ``point at infinity'' $z=-1$ 
or $z=1$. The trapped set is simply the projection of the two oriented geodesics in $S\hh^2$ relating $z=-1$ to $z=1$, and these are two oriented closed geodesics in the quotient. 

{\bf Experiment 1:} We pick the function in Fig. \ref{fig:model1} (right), compute its ray transform on Fig. \ref{fig:model1_recons} (left) using $500$ boundary points on each top and bottom boundary, and $1000$ directions for each boundary point uniformly within $\left( \frac{-\pi}{2}, \frac{\pi}{2} \right)$. We then apply the inversion formula, discretizing the $\sph^1$ integrals using $500$ directions. The error is visualized on Fig. \ref{fig:model1_recons} (right). 

\begin{figure}[htpb]
    \centering
    \includegraphics[trim= 20 0 20 0, clip, height=0.28\textheight]{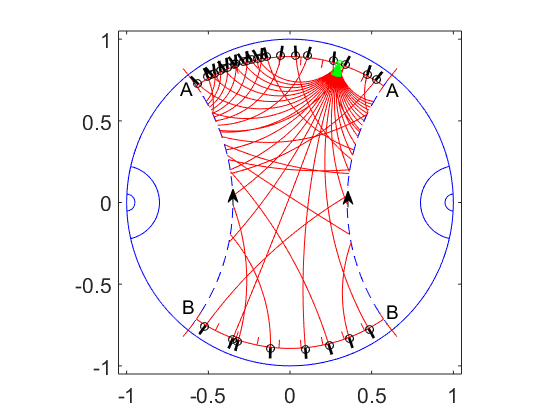}
    \includegraphics[trim= 40 0 40 0, clip, height=0.28\textheight]{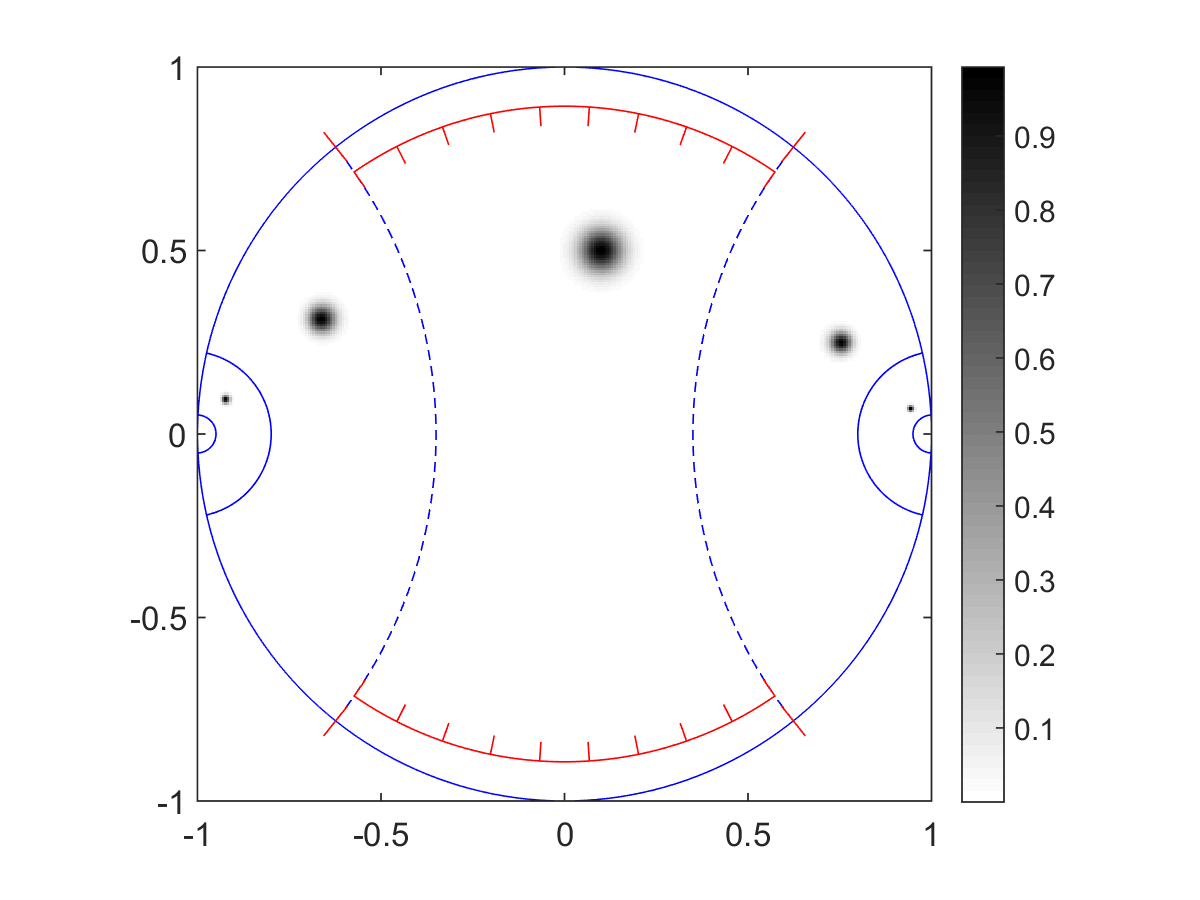}
    \caption{Left: the surface $M(-0.3)$ with a boundary cut out, with some geodesics cast from a boundary point superimposed. Right: example of a function whose ray transform is computed Fig. \ref{fig:model1_recons}}
    \label{fig:model1}
\end{figure}

\begin{figure}[htpb]
    \centering
    \includegraphics[trim=10 0 10 0, clip, height=0.28\textheight]{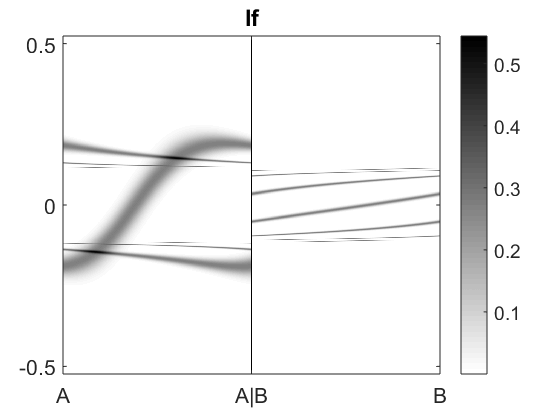}
    \includegraphics[trim=20 0 20 0, clip, height=0.28\textheight]{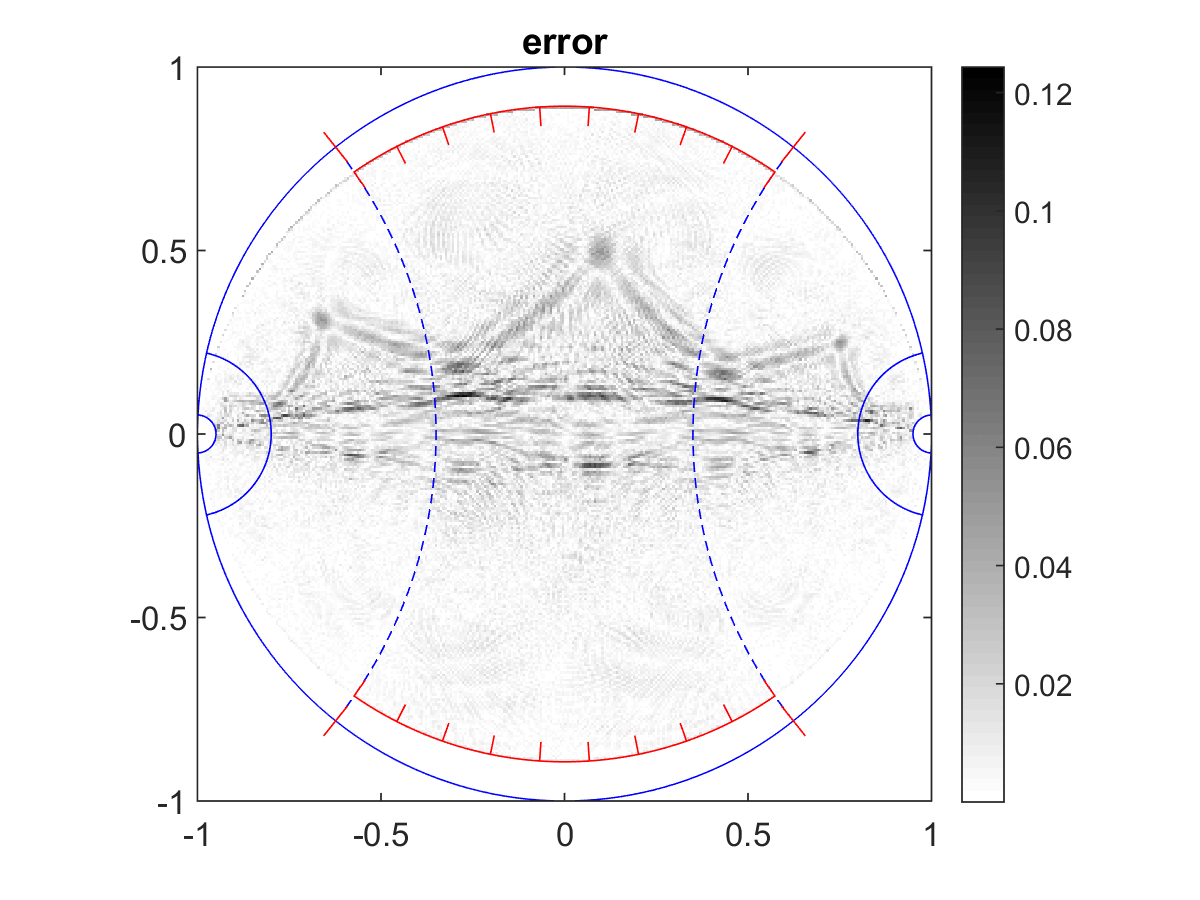}
    \caption{Experiment 1. Left/middle: ray transform of the function displayed in Fig \ref{fig:model1}. The $x$-axis represents a point on the boundary parameterized with a parameter in $[0,1]$, and the $y$-axis is the shooting direction with respect to the inner normal, an angle in $\left(\frac{-\pi}{2}, \frac{\pi}{2}\right)$. Right: pointwise error on reconstruction after applying the formula from Theorem \ref{inversion} to the data.}
    \label{fig:model1_recons}
\end{figure}

{\bf Observations:} Based on the results displayed on Fig. \ref{fig:model1_recons}, one can notice artifacts of 10\% relative size in certain areas. These artifacts lie in the region between the support of the function and the points at infinity (this is especially noticeable once we periodize the reconstruction), and we propose two possible reasons for this: 
\begin{itemize}
    \item As the computation of forward data does not use an exact computation of the geodesics, the accuracy degrades on the longer geodesics, i.e. the ones that wander near the closed one. 
    \item The data may have logarithmic singularities at trapped directions, which upon applying the Hilbert transform, may create undesirable oscillations in the vicinity of these angles. These oscillations then create artifacts upon applying the backprojection $I_1^*$, especially at the points which are such that the data is nonzero near their trapped directions (see Fig. \ref{fig:exp2_2} below). This is precisely the case when these points lie between the support of the function and the points at infinity. 
\end{itemize}

\subsection{Quotients of $\dd^2$ by Schottky groups with two generators}
We now consider examples of quotients of $\dd^2$ by Schottky groups with two generators. Keeping $x\in (-1,0)$, $a(x)$ and $T_a$ defined as above, we define the surfaces $M_1(x) = \dd^2 / \langle T_{a(x)}, T_{ia(x)} \rangle$ and $M_2(x) = \dd^2 / \langle i T_a, -i T_{ia} \rangle$. Both $M_1(x)$ and $M_2(x)$ can have the same fundamental domain, delimited by the geodesics $\gamma_{(\pm x, 0), \bfe_y}$ and $\gamma_{(0,\pm x), \bfe_x}$, though with different identifications, as represented Fig. \ref{fig:Schottky2}. 
In both examples, the trapping is now more complicated: at every point $\rmx$, the trapped directions form a fractal set at the tangent circle $S_\rmx$. In light of Experiment 1, this is expected to have quite a detrimental impact on the inversion. We study this influence by comparing noise on experiments 2 and 4, where the same model $M_1(x)$ is chosen for two different values of $x$, yielding different trapping intensities. We quantify this by the dimension of the limit set $\delta_\Gamma\in [0,1]$ of the group $\Gamma = \langle T_{a(x)}, T_{ia(x)} \rangle$. We compute a numerical approximation of $\delta_\Gamma$ by running geodesic dynamics on a large number of points inside the domain and measuring their escape rate, i.e. the proportion that remains inside as a function of time. According to \cite[Appendix B]{DyGu1}, this escape rate is of the form $V(t) = e^{-(1-\delta_\Gamma)t}$, and we use this to approximate $\delta_\Gamma$ by linearly interpolating $-\log V(t)$, possibly after throwing away transient regimes. In Experiments 2-3-4, the influx boundary is discretized uniformly into $1200$ boundary points and $400$ influx directions in the range $\left( -\frac{\pi}{6}, \frac{\pi}{6} \right)$ per boundary point.

{\bf Experiment 2.} On $M_1(-0.6)$, we pick the function in Fig. \ref{fig:exp2} (left), compute its ray transform on Fig. \ref{fig:exp2} (right) and apply the inversion formula. The reconstruction is visualized on Fig. \ref{fig:exp234} (left). For later comparison, we have computed $\delta_\Gamma \approx .49$ following the method described above.

{\bf Experiment 3.} On $M_2(-0.6)$, we pick the function in Fig. \ref{fig:exp3} (left), compute its ray transform on Fig. \ref{fig:exp3} (right) and apply the inversion formula. The reconstruction is visualized on Fig. \ref{fig:exp234} (middle).

{\bf Experiment 4.} We now repeat Experiment 2, this time on $M_1(-0.5)$, whose topology is the same, though the trapping is now more prominent (we have computed $\delta_\Gamma \approx .65$). We pick the function in Fig. \ref{fig:exp4} (left), compute its ray transform on Fig. \ref{fig:exp4} (right) and apply the inversion formula. The reconstruction is visualized on Fig. \ref{fig:exp234} (right).  

{\bf Observations.} As in Experiment 1, the same comments hold regarding the influence of trapping: the more trapping, the more singularities in the data, which create oscillations upon applying the Hilbert transform (see Fig. \ref{fig:exp2_2}), in turn yielding artifacts when applying the backprojection $I_1^*$. Comparing Experiments 2 and 4, the difference in trapping can be both seen on the singularities of the forward transforms (compare Figs. \ref{fig:exp2}, right and \ref{fig:exp4}, right), and on higher noise in reconstruction (compare Figs. \ref{fig:exp234}, left and right). 

\begin{figure}[htpb]
    \centering
    \includegraphics[trim= 30 0 20 0, clip, height=0.27\textheight]{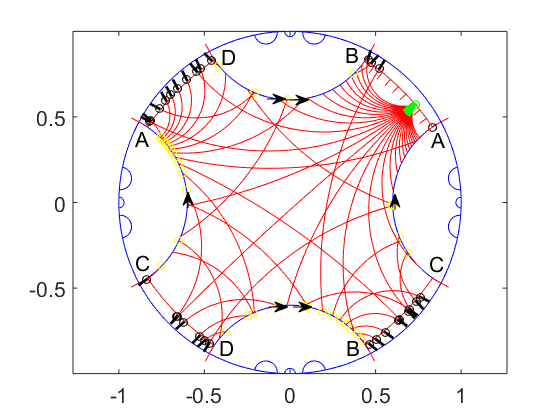}
    \includegraphics[trim= 20 0 20 0, clip, height=0.27\textheight]{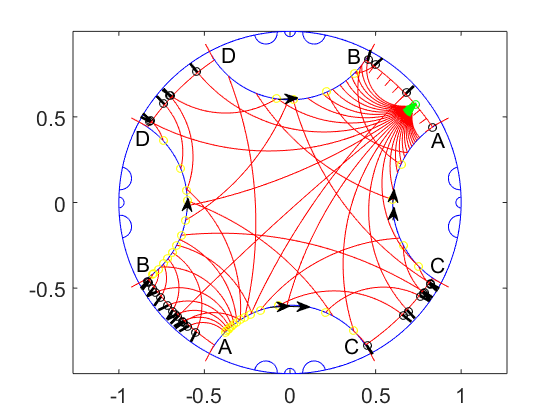}
    \caption{Quotients of $\dd$ by Schottky groups with two generators. Left: ``torus with boundary'' $M_1(-0.5)$ ($\partial M_1$ has one connected component). Right: ``pair of pants'' $M_2(-0.5)$ ($\partial M_2$ has three connected components). Identifications are marked with similar arrow patterns and the same letter identifies boundary points together. In both cases, a fan of geodesics is cast from the point on the top-right and plotted inside the fundamental domain until they exit.}
    \label{fig:Schottky2}
\end{figure}

\begin{figure}[htpb]
    \centering
    \includegraphics[trim= 30 0 40 10, clip, height=0.23\textheight]{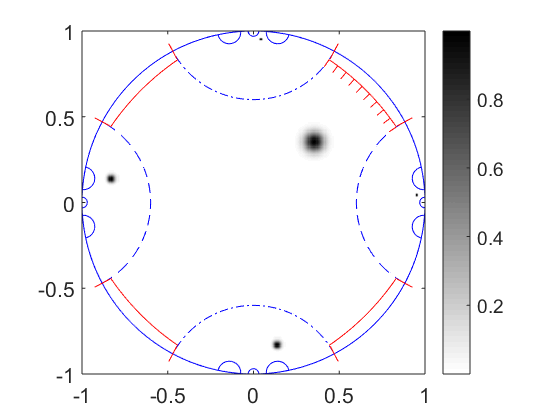}
    \includegraphics[trim= 30 0 40 25, clip, height=0.23\textheight]{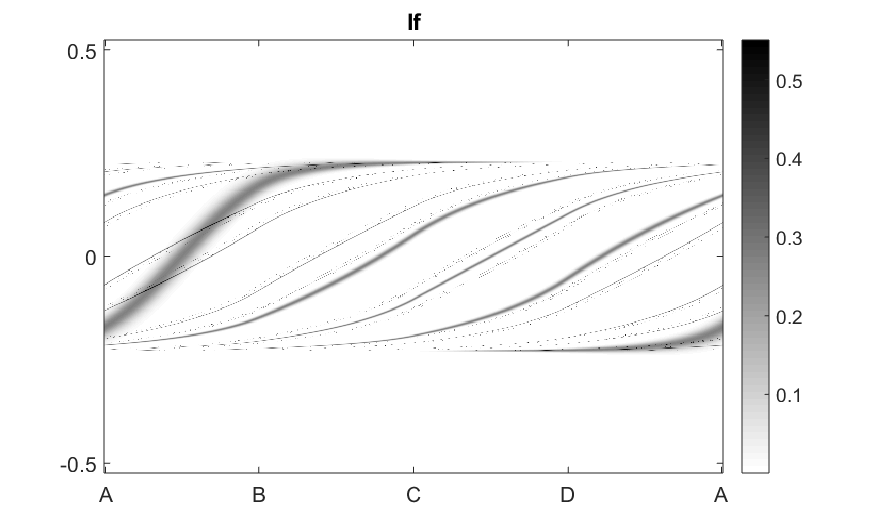}
    \caption{Experiment 2. Left: the function $f$ on $M_1(-0.6)$. Right: its ray transform $I_0 f$.}
    \label{fig:exp2}
\end{figure}

\begin{figure}[htpb]
    \centering
    \includegraphics[trim= 30 0 40 0, clip, height=0.23\textheight]{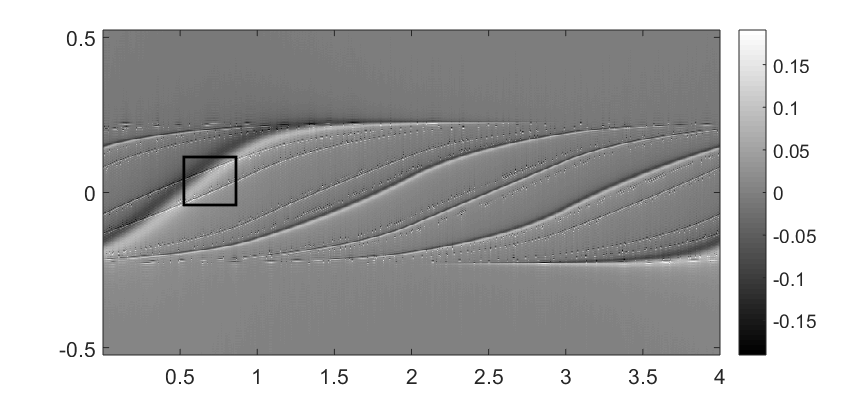}
    \includegraphics[trim= 30 0 40 0, clip, height=0.23\textheight]{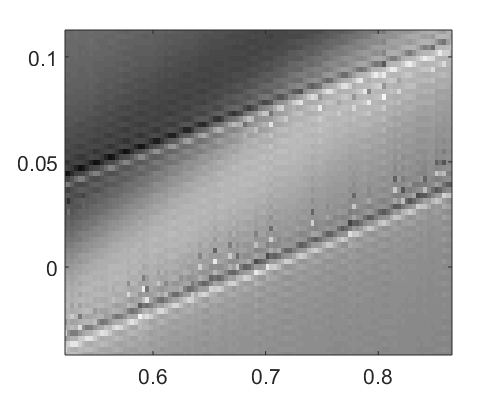}
    \caption{Experiment 2. Left: the Hilbert transform $H I^{od}_0 f$ of the data $I_0 f$ displayed on Fig. \ref{fig:exp2}, right, restricted to $\partial_+ SM$. Right: zoom into the boxed region on the left, emphasizing the oscillations which appear after applying the Hilbert transform to data with singularities.}
    \label{fig:exp2_2}
\end{figure}

\begin{figure}[htpb]
    \centering
    \includegraphics[trim= 30 0 40 10, clip, height=0.23\textheight]{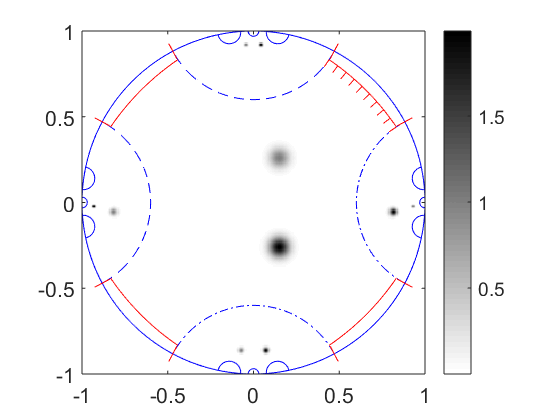}
    \includegraphics[trim= 30 0 40 25, clip, height=0.23\textheight]{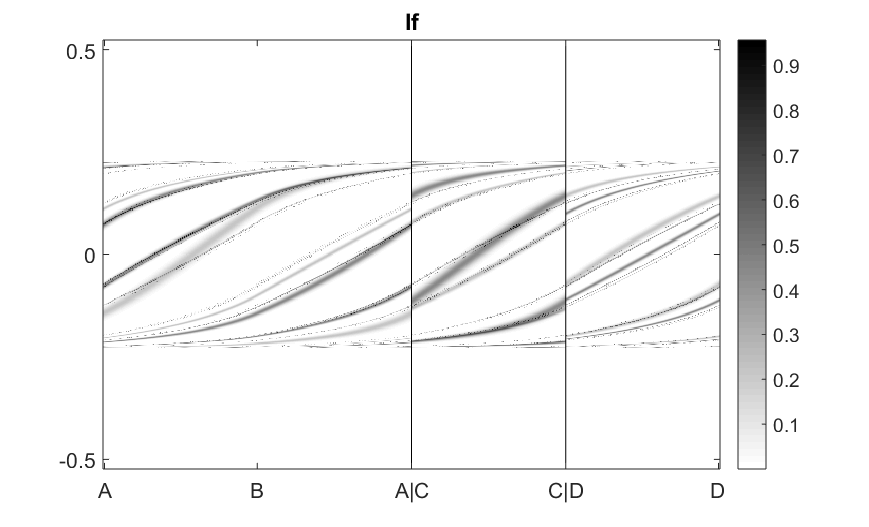}
    \caption{Experiment 3. Left: the function $f$ on $M_2(-0.6)$. Right: its ray transform $I_0f$, supported on three connected components, see Fig. \ref{fig:Schottky2} (right) for the letter positions.}
    \label{fig:exp3}
\end{figure}

\begin{figure}[htpb]
    \centering
    \includegraphics[trim= 30 0 40 10, clip, height=0.23\textheight]{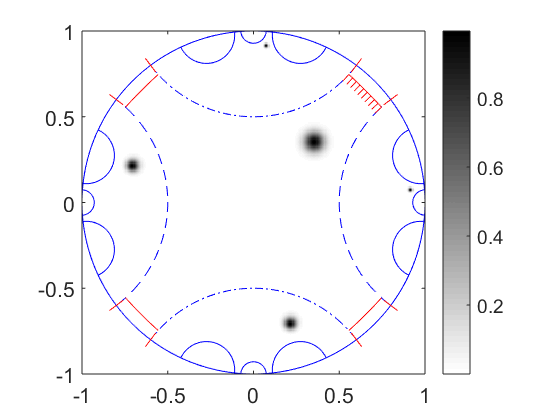}
    \includegraphics[trim= 30 0 40 25, clip, height=0.23\textheight]{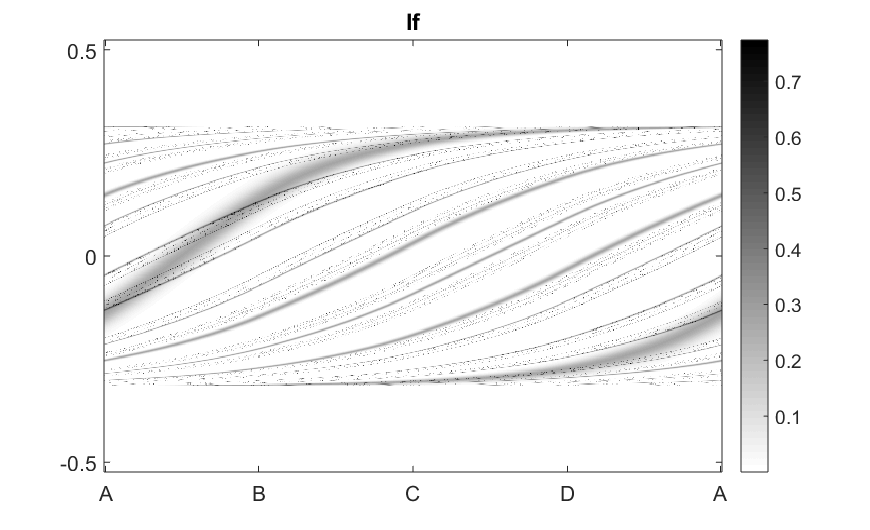}
    \caption{Experiment 4. Left: the function $f$ on $M_1(-0.5)$. Right: its ray transform $I_0 f$.}
   \label{fig:exp4}
\end{figure}

\begin{figure}[htpb]
  \centering
    \includegraphics[trim= 30 0 40 15, clip, width=0.325\textwidth]{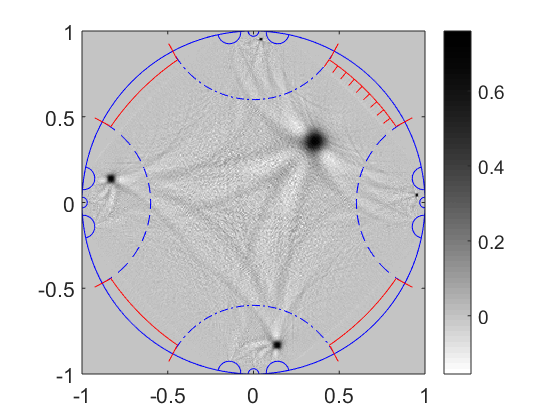}
    \includegraphics[trim= 30 0 40 15, clip, width=0.325\textwidth]{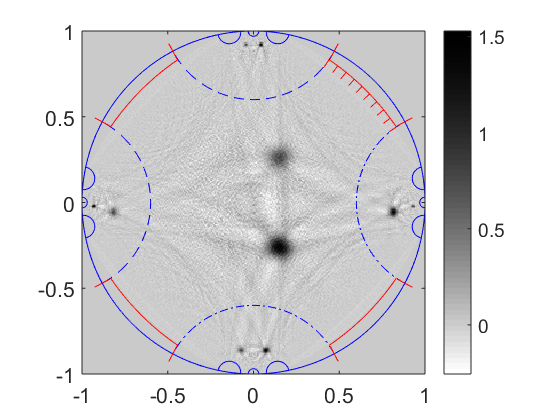}
    \includegraphics[trim= 30 0 40 15, clip, width=0.325\textwidth]{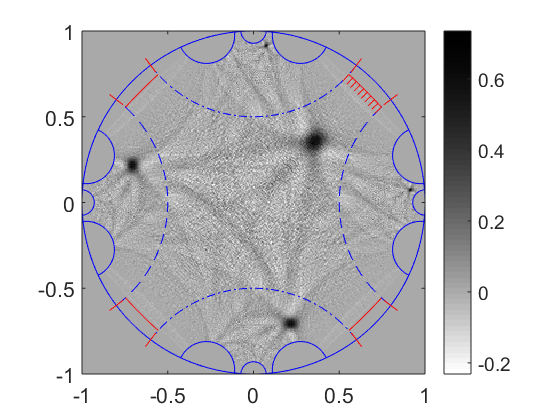}
  \caption{Reconstructions for Experiments 2 (left), 3 (middle) and 4 (right). The reconstructions have been periodized to make appear the structure in error patterns, supported between the gaussians and all points at infinity (note: color scheme is darker for visibility).}
  \label{fig:exp234}
\end{figure}

\subsection{Cylinders with variable curvature}

Let the topological cylinder $\rr/2\zz \times (-1,1)$ endowed with the metric $g(x,y) = h_\eps(y)^2 dx^2 + dy^2$, where we have defined
\[ h_\eps(y) = \cosh (y) \cosh (\eps y), \quad \eps >0\quad (\text{constant}). \]
The Gauss curvature is independent of $x$ and given by 
\[ \kappa_\eps(y)= -\frac{\pl_y^2 h_\eps(y)}{h_\eps(y)}=-1-\eps^2-2\eps \tanh(y) \tanh(\eps y), \]
so we have $-(1+\eps)^2 \le \kappa_\eps(y) \le -(1+\eps^2)$ for every $y$, and we also get 
\begin{align*}
    |d\kappa_\eps(y)| \le 2\eps ( (1-\tanh^2(y))|\tanh(\eps y)| + \eps |\tanh(y)| (1-\tanh^2(\eps y)) \le 2\eps(1+\eps). 
\end{align*}

The structure of the trapped set is similar to that of Experiment 1, with the set $\{y=0\}$ the projection on the base of the only closed geodesic and $\delta_\Gamma = 0$. We can bound $g$ in terms of the constant curvature metric $g_0$ with $\kappa_0 = 1$ via the constants $\lambda_1 = \frac{1}{\cosh\eps}$ and $\lambda_2 = 1$ as in Theorem \ref{normL^2Pi0}. The requirement $\lambda_1\lambda_2 > \max (\delta_\Gamma, \frac{1}{2})$ there gives the constraint $\eps \le \cosh^{-1} 2$. Since $\delta_\Gamma \le \frac{1}{2}$, we deduce the estimate, with $\lambda = 1-\frac{1}{\cosh \eps}$, 
\begin{align*}
    \|\Pi_0\|_{L^2(M,g)\to L^2(M,g)} \le 4 \cosh^4 \eps \left( \frac{\Gamma \left( \frac{1}{2} \left( \frac{1}{\cosh \eps} - \frac{1}{2} \right) \right) \Gamma \left( \frac{3}{4} - \frac{1}{2\cosh\eps} \right)}{\Gamma\left( \frac{1}{4} \right) \Gamma \left( \frac{3}{4} \right)} \right)^2.  
\end{align*}
Combining this estimate with the estimate in Proposition \ref{normW}, we deduce 
\begin{align*}
    \|W\|_{L^2\to L^2} \le \frac{8}{3} \eps (1+\eps) \cosh^4 \eps \left( \frac{\Gamma \left( \frac{1}{2} \left( \frac{1}{\cosh \eps} - \frac{1}{2} \right) \right) \Gamma \left( \frac{3}{4} - \frac{1}{2\cosh\eps} \right)}{\Gamma\left( \frac{1}{4} \right) \Gamma \left( \frac{3}{4} \right)} \right)^2.
\end{align*}
In particular, $\lim_{\eps \to 0} \|W\|_{L^2\to L^2} = 0$ and thus for $\eps$ small enough, inversion via a Neumann series is justified. As in \cite{Mo}, when implementing the Neumann series, the error operator $-W^2$ is not implemented directly as it would be hopeless to satisfy \eqref{formulainv0} at the discrete level. Instead, once $I_0$ and its approximate inverse (call it $A_0$, say) are discretized, we set $-W^2 = Id - A_0 I_0$ and iterate the Neumann series using this operator.

\begin{figure}[htpb]
  \centering
  \includegraphics[trim= 10 0 30 10, clip, height=0.23\textheight]{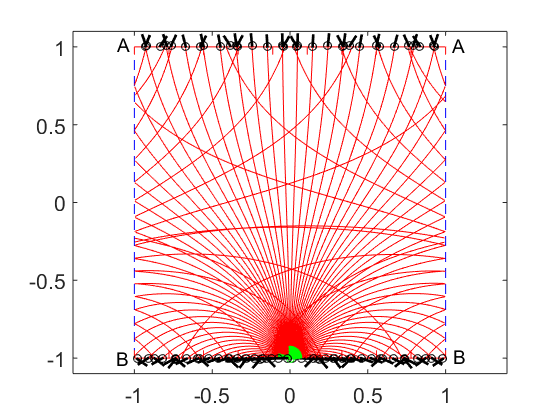} 
  \includegraphics[trim= 10 0 30 10, clip, height=0.23\textheight]{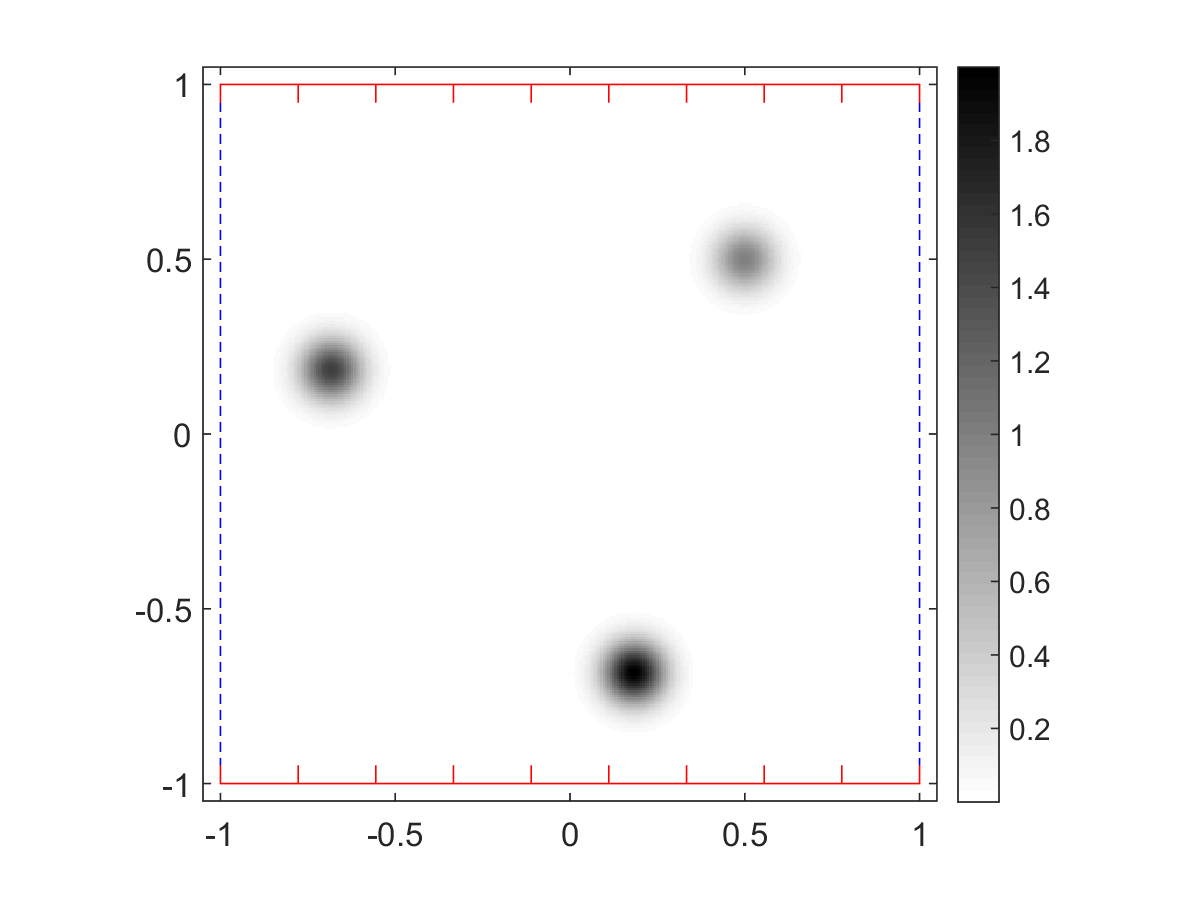}
  \caption{Experiment 5. Left: model for cylinder with variable curvature with constant $\eps = 0.4$. Right: Example of function to be imaged.}
  \label{fig:exp5}
\end{figure}

\begin{figure}[htpb]
  \centering
  \includegraphics[trim= 10 0 30 20, clip, height=0.23\textheight]{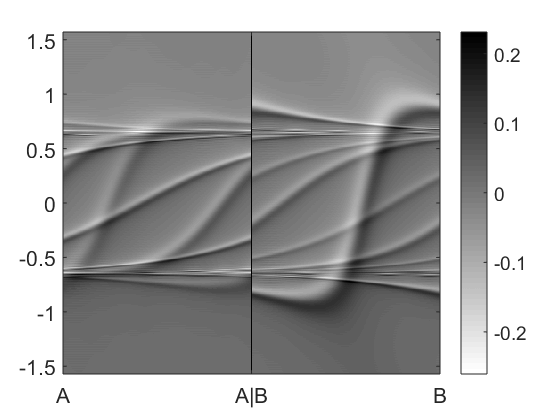}
  \includegraphics[trim= 10 0 30 28, clip, height=0.23\textheight]{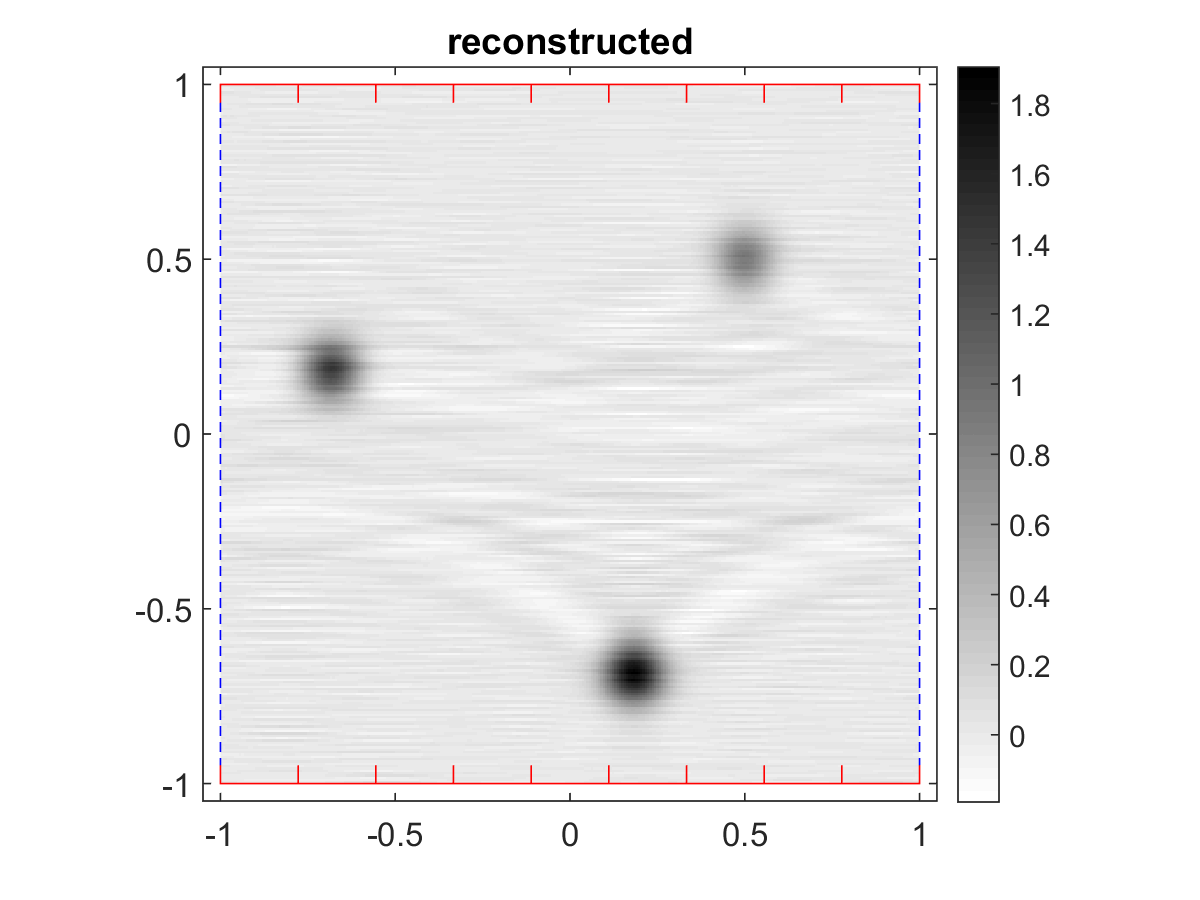}
  \includegraphics[trim= 10 0 30 26, clip, height=0.23\textheight]{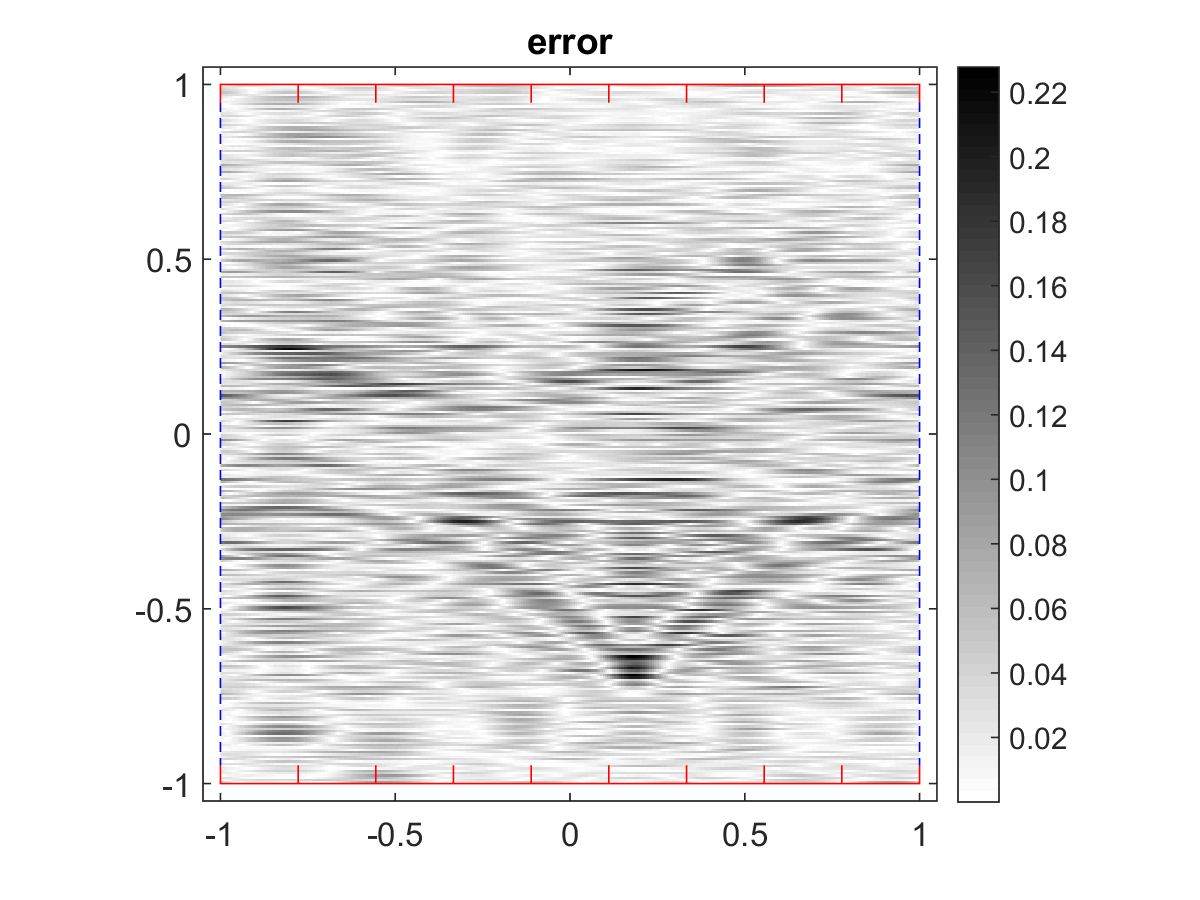}
  \includegraphics[trim= 10 0 30 26, clip, height=0.23\textheight]{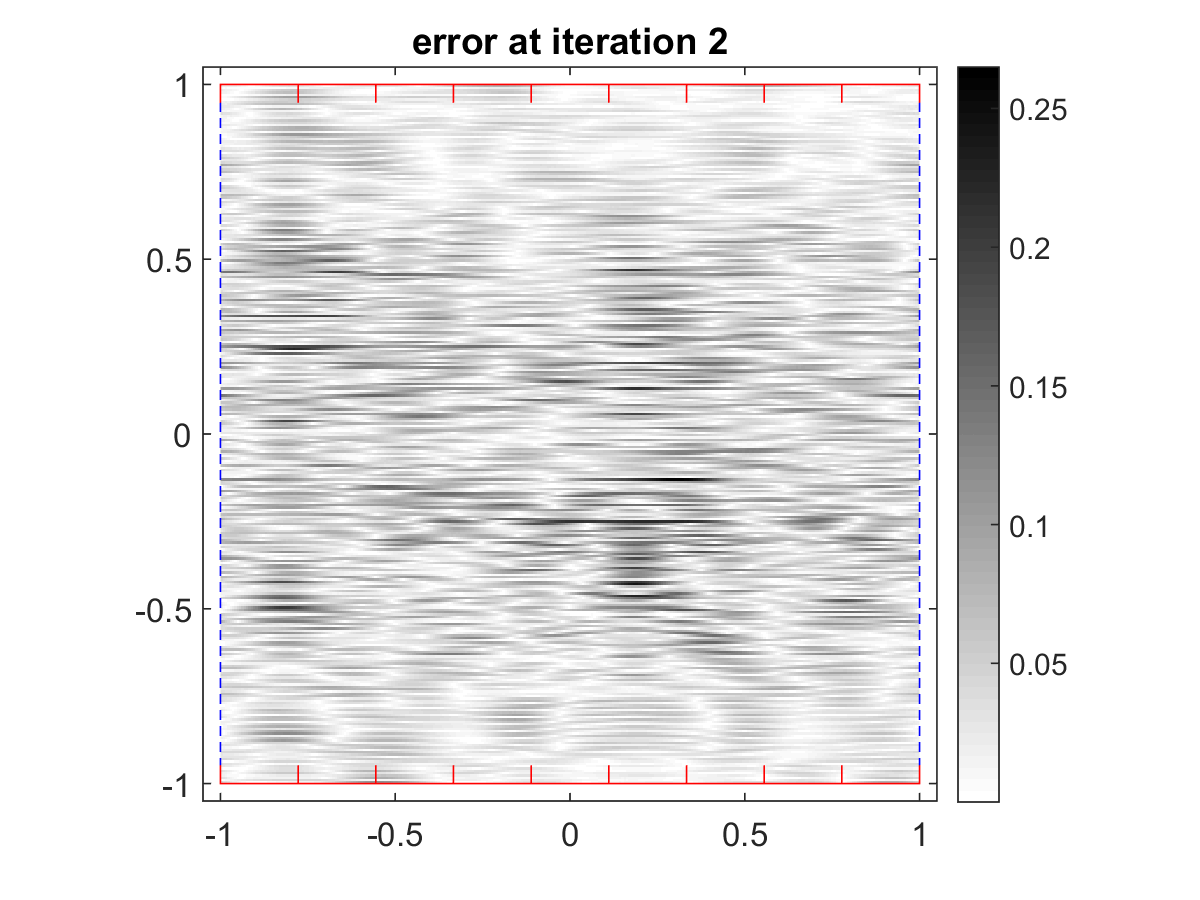}
  \caption{Experiment 5. Top-left: data $I_0 f$. Top-right: reconstructed $f$ after one-shot inversion. Bottom: pointwise error after one-shot inversion (left), then after two Neumann series iterations (right).}
  \label{fig:exp5_2}
\end{figure}

{\bf Experiment 5.} We pick the function in Fig. \ref{fig:exp5} (right), compute its ray transform on Fig. \ref{fig:exp5_2} (top-left) and implement the approximate inversion formula (visualized Fig. \ref{fig:exp5_2}, top-right), followed by two Neumann series iterations. Pointwise errors after both one-shot and corrected reconstructions are visualized Fig. \ref{fig:exp5_2} (bottom).

\subsection{Discussion}

In light of these experiments, we make the following comments: 

\begin{itemize}
  \item Despite the theoretical exactness and stability of the proposed reconstruction formulas, geometries with trapping pose some numerical challenges, related to (i) the singularities it creates in data space, which in turn create oscillations after applying the Hilbert transform, yieling reconstruction artifacts after backprojection; and (ii), the loss of accuracy on the computation of long geodesics which wander near the trapped set before exitting the domain. 
  \item As mentioned before in \cite{Mo}, another important challenge in surfaces with nonconstant curvature is to understand how to design an appropriate discretization of the influx boundary whose flow-out samples the manifold in as uniform a way as possible. 
  \item In the case of non-constant curvature, iterative corrections, which require a differentiation step at every iteration, may increase noise at high frequencies (which, as mentioned earlier, becomes more prominent in the presence of trapping, unlike the case of simple surfaces covered in \cite{Mo}). Regularization methods should then be considered. 
\end{itemize}

%\newpage
%$ $ 
%\newpage

\appendix 
\section{Computation of $\Pi_0^\la$ in constant curvature.}

Let $M:=\Gamma\backslash \hh^2$ be a quotient of hyperbolic space by a convex co-compact group $\Gamma\subset {\rm PSL}_2(\rr)$ and denote by $\delta_\Gamma$ the exponent of the group $\Gamma$ as defined 
by \eqref{deltaGamma}. Such a manifold is non-compact but can be compactified into $\bbar{M}=\Gamma\backslash (\hh^2\cup \Omega_{\Gamma})$ where $\Omega_{\Gamma}=\mathbb{S}^1\setminus \Lambda_{\Gamma}$ is the set of discontinuity of $\Gamma$; $\hh^2\cup \Omega_\Gamma$ is the largest open set in the closed unit disk $\bbar{\mathbb{D}^2}$ where $\Gamma$ acts properly discontinuously 
(we view $\hh^2$ as the unit disk $\mathbb{D}^2$ in $\cc$).
The Laplacian $\Delta_M$ has continuous spectrum $[\frac{1}{4},\infty)$ on $L^2(M)$ and possibly finitely many eigenvalues in $(0,\frac{1}{4})$, which appear if and only if the exponent $\delta_{\Gamma}$ of $\Gamma$ satisfies $\delta_{\Gamma}>1/2$, in which case $\delta_{\Gamma}(1-\delta_{\Gamma})=\min ({\rm Spec}_{L^2}(\Delta_M))$. For further details on these topics, we refer to the book \cite{Bo}. 

In this Appendix, we compute for small $(1-\la)>\max (\delta_\Gamma,\demi)$ the operator on $M$ 
\begin{equation}\label{Pi0lambda}
\Pi_0^\la : C_c^\infty(M)\to C^\infty(M), \quad \Pi_0^\la f(x):=\int_{S_xM}\int_\rr 
\cosh(\la t)\pi_0^*f(\varphi_t(x,v))dt dS_{x}(v)
\end{equation} 
where $\varphi_t$ denotes the geodesic flow at time $t$ on the unit tangent bundle $SM$. Just like in Section 
\ref{normofPi0}, the operator $\Pi_0^\la$ makes sense and 
has integral kernel $\mc{K}^\la_M$, which can written as a converging sum on a fundamental domain $F\subset \hh^2$ of $\Gamma$ when $(1-\la)>\delta_\Gamma$
\begin{equation}\label{KlaM}
\mc{K}^{\la}_M(x,x')=\sum_{\gamma\in \Gamma}\mc{K}^{\la}(\gamma\tilde{x},\tilde{x}')),\quad x=\pi_{\Gamma}(\tilde{x}),\, x'=\pi_{\Gamma}(\tilde{x})\end{equation}
for $\tilde{x},\tilde{x}'\in F$ and $\pi_\Gamma:\hh^2\to M$  the covering map, where $\mc{K}^\la(\tilde{x},\tilde{x}')$ is the integral kernel of the operator $\til{\Pi}_0^\la$ defined by the expression 
\eqref{Pi0lambda} when $M=\hh^2$. As in Section \ref{normofPi0}, this kernel is given by 
\[\mc{K}(\tilde{x},\tilde{x}')=\frac{2\cosh(\la d_{\hh^2}(\tilde{x},\tilde{x}'))}{\sinh(d_{\hh^2}(\tilde{x},\tilde{x}'))}\]
which justifies that the sum \eqref{KlaM} above converges if $(1-\la)>\delta_\Gamma$.
The expression \eqref{Pi0lambda} can also be rewritten as 
\[\Pi_0^\la f(x):=\int_{S_xM}\int_\rr 
e^{\la t}\pi_0^*f(\varphi_t(x,v))dt dS_{x}(v)\]
using the symmetry $v\to -v$ in each fiber $S_xM$.
\begin{lemm}\label{Pi0hyp} 
On a convex co-compact quotient  $M:=\Gamma\backslash \hh^2$ with exponent of convergence given by $\delta_\Gamma$, the operator $\Pi_0^\la$ can be written when $(1-\la)>\max (\delta,\demi)$ as  
\[ \Pi_0^\la = 4\frac{\Gamma(\frac{1-2\la}{4}-\frac{i\sqrt{\Delta_{M}-\frac{1}{4}}}{2})\Gamma(\frac{1+2\la}{4}-\frac{i\sqrt{\Delta_{M}-\frac{1}{4}}}{2})\Gamma(\frac{1+2\la}{4}+\frac{i\sqrt{\Delta_{M}-\frac{1}{4}}}{2})\Gamma(\frac{1-2\la}{4}+\frac{i\sqrt{\Delta_{M}-\frac{1}{4}}}{2})}{\Gamma(\frac{1}{4}-\frac{i\sqrt{\Delta_{M}-\frac{1}{4}}}{2})\Gamma(\frac{1}{4}+\frac{i\sqrt{\Delta_{M}-\frac{1}{4}}}{2})\Gamma(\frac{3}{4}-\frac{i\sqrt{\Delta_{M}-\frac{1}{4}}}{2})\Gamma(\frac{3}{4}+\frac{i\sqrt{\Delta_{M}-\frac{1}{4}}}{2})}\]
where $\Gamma(\cdot)$ is the Euler Gamma function, and with the convention that 
$\sqrt{s(1-s)-\frac{1}{4}}=i(s-\demi)$ when $s(1-s)\in(0,\frac{1}{4})$ is an $L^2$-eigenvalue of $\Delta_{M}$ with $s\in (1/2,1)$.
The operator norm is given by 
\begin{equation}\label{bound1}\begin{gathered}
||\Pi_0^\la||_{L^2\to L^2}= 4\Big(\frac{\Gamma(\frac{1-2\la}{4})\Gamma(\frac{1+2\la}{4})}{\Gamma(\frac{1}{4})\Gamma(\frac{3}{4})}\Big)^2 \textrm{ if }\delta_{\Gamma}\leq 1/2,\\
||\Pi_0^\la||_{L^2\to L^2}= 4\frac{\Gamma(\frac{\delta_{\Gamma}-\la}{2})\Gamma(\frac{\delta_{\Gamma}+\la}{2})\Gamma(\frac{1-\la-\delta_{\Gamma}}{2})\Gamma(\frac{1+\la-\delta_{\Gamma}}{2})}
{\Gamma(\frac{1-\delta_{\Gamma}}{2})\Gamma(\frac{\delta_\Gamma}{2})\Gamma(1-\frac{\delta_{\Gamma}}{2})\Gamma(1+\frac{\delta_{\Gamma}}{2})}  
\textrm{ if }\delta_{\Gamma}\geq 1/2.
\end{gathered}\end{equation}
\end{lemm}
\begin{proof} 
Since its Schwartz kernel is a function of $d_{\hh^2}(\cdot,\cdot)$ decaying like $e^{-d}$, 
the operator $\til{\Pi}_0^\la$ is a function $H_\la(\Delta_{\hh^2}-\tfrac{1}{4})$ 
of the hyperbolic Laplacian $\Delta_{\hh^2}$ (see for example \cite[Section 1.8]{Iw}).
Notice that the series defining $\mc{K}^\la_M$ is absolutely convergent  on compact sets of $\hh^2\x\hh^2$ 
since $\delta_{\Gamma}<1$ if $\Gamma$ is convex co-compact.
The operator $\til{\Pi}^\la_0$ is bounded on $L^2(\hh^2)$ if $\la<1/2$: indeed, this is an operator in the class 
$\Psi_0^{-1,(1-\la),(1-\la)}(\hh^2)$ of the $0$-calculus of pseudo-differential operators introduced by Mazzeo-Melrose \cite{MaMe} and the $L^2$-boundedness on $\hh^2$ follows for instance from Mazzeo 
\cite[Theorem 3.25]{Ma}. 
To compute the function $H_\la$, we use the representation theory of $G:={\rm SL}_2(\rr)$, following \cite{La,AnZe}. The (time reversal) principal series representations are unitary representations into $L^2(\rr,dx)$ defined as follows 
\[ \mc{P}_{ir}: 
\begin{bmatrix}
    a   & b \\
    c & d     
\end{bmatrix}
\in G\mapsto  \Big(f(x)\mapsto |-bx+d|^{-1-2ir}f\Big(\frac{ax-c}{-bx+d}\Big)\Big).
\]
Let $u\in C_c^\infty(\hh^2)$, then $u_0:=\pi_0^*u\in C_c^\infty(G)$ (if
we identify $S\hh^2\simeq G$ in the canonical way) and  to compute $\til{\Pi}_0f$ we will use 
the Plancherel formula of Harish-Chandra \cite{Ha}  
\[ u_0(1)=\int_{\rr} r\tanh(\tfrac{\pi r}{2}) {\rm Tr}\Big(\int_{G}u_0(g)\mc{P}_{ir}(g) dg\Big)dr.\]
There is a unique $K$-invariant vector of norm $1$ in $L^2(\rr)$ where $K$ is the action of 
${\rm SO}(2)\subset {\rm SL}_2(\rr)$ via the representation $\mc{P}_{ir}$, it is given by 
\[ \phi_{ir}(x)=\frac{1}{\sqrt{\pi}}(1+x^2)^{-\demi-ir}.\]
Since $u_0$ is invariant by ${\rm SO}(2)$, to consider the action of the right multiplication by an 
element $e^{tA}$  on $u_0$ with $A\in sl_2(\rr)$, it suffices to consider the action of $\mc{P}_{ir}(e^{tA})$ on $\phi_{ir}$.
The pull-back by the geodesic flow at time $t$ corresponds, in each  
unitary representation $\mc{P}_{ir}$, to the unitary operator  on $L^2(\rr)$ 
 \[ \mc{P}_{ir}(e^{tX}):  f(x)\mapsto e^{t(\demi+ir)}f(e^tx),\quad \textrm{ if }X=\begin{bmatrix}
    \demi   & 0 \\
    0 & -\demi   
\end{bmatrix}\in sl_2(\rr).\] 
The operator $\til{\Pi}^\la_0$ can be written as 
\[ \til{\Pi}^\la_{0}u(x)=\frac{1}{2\pi}\int_{0}^{2\pi}\Big(\int_{-\infty}^\infty e^{\la t}
u_0(h_xe^{tX}e^{sV})\,  dt\Big) ds \textrm{ if }V:=
\begin{bmatrix}
    0   & -1 \\
    1 & 0   
\end{bmatrix},\]
where $h_x\in {\rm SL}_2(\rr)$ is any element which projects to 
$x\in \hh^2={\rm SL}_2(\rr)/{\rm SO}(2)$.
Since the average in the compact group ${\rm SO}(2)$ amounts to project onto $\phi_{ir}$, 
the operator $\til{\Pi}_0^\la$ is unitarily equivalent 
to multiplication on the representation space $L^2(\rr)$ by
\[\int_{-\infty}^\infty \phi_{-ir}(x)\int_{-\infty}^\infty e^{t(\demi+\la+ir)}\phi_{ir}(e^{t}x)dt dx\] 
through the representation $\mc{P}_{ir}$.
By changing variable and by parity in $x$, this term is of the form $\frac{2}{\pi}F_\la(r)F_{-\la}(-r)$ with 
\[ \begin{split}
F_\la(r):=\int_{0}^{\infty}x^{-\demi-ir-\la}(1+x^2)^{-\demi+ir}dx=&
\frac{\Gamma(\frac{1}{4}-\frac{ir}{2}-\frac{\la}{2})
\Gamma(\frac{1}{4}-\frac{ir}{2}+\frac{\la}{2})}{\Gamma(\demi-ir)}\\
=&\sqrt{2\pi}2^{ir}\frac{\Gamma(\frac{1}{4}-\frac{ir}{2}-\frac{\la}{2})
\Gamma(\frac{1}{4}-\frac{ir}{2}+\frac{\la}{2})}{\Gamma(\frac{1}{4}-\frac{ir}{2})\Gamma(\frac{3}{4}-\frac{ir}{2})}. 
\end{split}\]
Notice that $F_\la(r)=F_{-\la}(r)$ and that $\frac{2}{\pi}F_\la(r)F_{-\la}(-r)$ extends holomorphically in each strip $|{\rm Im}(r)|<1/2-\la$ 
and it can be written under the form $H_\la(r^2)=\frac{2}{\pi}F_\la(r)F_\la(-r)$ for some $H_\la$ holomorphic 
in the region containing $\rr^+$ and bounded by the parabola ${\rm Re}(z)=({\rm Im}(z))^2-(1/2-\la)^2$.
The operator $\til{\Pi}_0^\la$ can then be written $\til{\Pi}^\la_0= H_\la(\Delta_{\hh^2}-\tfrac{1}{4})$ with
\[H_\la(z)=
4\frac{\Gamma(\frac{1}{4}-\frac{\la}{2}-\frac{i\sqrt{z}}{2})\Gamma(\frac{1}{4}+\frac{\la}{2}+\frac{i\sqrt{z}}{2})
\Gamma(\frac{1}{4}-\frac{\la}{2}+\frac{i\sqrt{z}}{2})\Gamma(\frac{1}{4}+\frac{\la}{2}-\frac{i\sqrt{z}}{2})}{\Gamma(\frac{3}{4}-\frac{i\sqrt{z}}{2})^2\Gamma(\frac{3}{4}+\frac{i\sqrt{z}}{2})^2}.\]
The expression of $\til{\Pi}^{\la}_0$ for $\la=0$ matches with the formula found by 
\cite[p. 623]{BeCa}. 
In the image of the strip $|{\rm Im}(r)|<1/2-\la -\eps$ by $r\mapsto z=r^2$  for $\eps>0$ small enough
\[ |H_\la (z)|\leq C_{\eps}(1+|z|)^{-1/2}\]
for some $C_\eps>0$ depending only on $\eps$.  We can then write, using Cauchy formula
\begin{equation}\label{integralcontour} 
\til{\Pi}_0^\la= \frac{1}{2\pi i}\int_{\mc{C}_\eps} H_\la(z)R_{\hh^2}(z) dz ,\quad R_{\hh^2}(z):=(\Delta_{\hh^2}-\tfrac{1}{4}-z)^{-1}
\end{equation}
where $\mc{C}_\eps$ is the contour given by the image of ${\rm Im}(r)=\demi-\la-\eps$ by the map $r\mapsto z=r^2$, where $\eps>0$ is chosen small.
The contour is parametrized by $z=\pm i(1-\la-2\eps)\sqrt{t}+t-(\demi-\la-\eps)^2$ with $t\in [0,\infty)$ and the 
$H^1(\hh^2)\to L^2(\hh^2)$ norm of $R_{\hh^2}(z)$ is 
\[ ||R_{\hh^2}(z)||_{H^1\to L^2}= \mc{O}((1+t)^{-1})\] 
showing that the integral \eqref{integralcontour} converges in $H^1\to L^2$ norm.
Now the resolvent  $R_M(r^2)=(\Delta_M-1/4-r^2)^{-1}$ has for integral kernel, viewed on the fundamental domain $F$, 
\[ R_M(r^2; \tilde{x},\tilde{x}')=\sum_{\gamma\in \Gamma} R_{\hh^2}(r^2; \tilde{x},\gamma\tilde{x}')\]
when ${\rm Im}(r)>\delta_\Gamma-1/2$, and the sum is uniformly convergent on compact sets of $F\x F$
for $r$ in that half-space: the proof is done in \cite{Pat2} or \cite[Proposition 2.2]{GuNa}, and follows from the uniform convergence of Poincar\'e series $\sum_{\gamma\in \Gamma}e^{-sd_{\hh^2}(\tilde{x},\gamma\tilde{x}')}$
for $\tilde{x},\tilde{x}'$ on compact sets and from the
 pointwise estimate for $\eps \leq {\rm Im}(r)\leq \demi-\eps$ (which follows from the explicit formula for $R_{\hh^2}(r^2)$ given for instance in \cite{GuZw})
\[ |R_{\hh^2}(r^2;\tilde{x},\tilde{x}')|\leq C_{\eta,\eps}e^{-(1/2+{\rm Im}(r))d_{\hh^2}(\tilde{x},\tilde{x}')} \quad 
\textrm{when } d_{\hh^2}(\tilde{x},\tilde{x}')>\eta\]
if $\eta>0$ is fixed and $C_{\eta,\eps}$ depends only on $\eta$ and $\eps$. This implies that, for $K_1$ and $K_2$ two relatively compact open sets of $\hh^2$ so that $d_{\hh^2}(K_1,K_2)>1$, 
there is $C>0$ depending only on the volume of $K_1$ and $K_2$ 
such that for all $f\in L^2(K_1)$
\[ ||R_{\hh^2}(r^2)f||_{L^2(K_2)}\leq Ce^{-(1/2+{\rm Im}(r))d_{\hh^2}(K_1,K_2)}||f||_{L^2(K_1)}.\]
Now, using that for $f\in L^2$ supported in a compact set $K_1$ we have  
\[ (\Delta_{\hh^2}+\tfrac{3}{4})\chi R_{\hh^2}(r^2)f-[\Delta_{\hh^2},\chi]R_{\hh^2}(r^2)f=
(1+r^2)\chi R_{\hh^2}(r^2)f\]
for each $\chi\in C_c^\infty(\hh^2)$ supported in a relatively compact set $K_2$ 
with $K_2\cap K_1=\emptyset$, we get 
\[ \begin{split}
||\chi R_{\hh^2}(r^2)f||_{H^{-2}(K_2)}\leq &\frac{C'}{|r|^2-1}||R_{\hh^2}(r^2)f||_{L^2(K_2)}\\
&\leq \frac{CC'}{|r|^2-1}e^{-(\demi+{\rm Im}(r))d_{\hh^2}(K_1,K_2)}||f||_{L^2(K_1)}
\end{split}\]
for $|r|>1$ where $C'$ depends only on $||\chi||_{C^1}$. This estimate together with the Poincar\'e series convergence implies that for each relatively compact open set $K\subset F$,  
$\tilde{x},\tilde{x}'\in K$ and $\pi_\Gamma(\tilde{x})=x,\pi_{\Gamma}(\tilde{x}')=x'$,
\[\Pi_0^\la(x,\gamma x')= \frac{1}{2\pi i}\int_{\mc{C}_\eps} H_\la(z)\sum_{\gamma\in\Gamma}R_{\hh^2}(z,\tilde{x},\gamma\tilde{x}') dz =\frac{1}{2\pi i}\int_{\mc{C}_\eps} H_\la(z)R_{M}(z,x,x') dz \]
converges in the sense of Schwartz kernels of bounded operator $L^2(K)\to H^{-2}(K)$. 
This shows that for $(1-\la)>\max(\demi,\delta)$ 
\[\Pi^\la_0 = H_\la(\Delta_M-1/4)\]
where $H_\la$ is the holomorphic function defined before. Now the bottom of the spectrum of $\Delta_M$
corresponds to $r=i\max(\delta_{\Gamma}-\demi, 0)$. Let us show that the function 
$\theta: t\in [0,\infty)\mapsto H_\la(4t^2)\in (0,\infty)$ is decreasing: we have 
\[ \frac{\theta'(t)}{\theta(t)}=-2({\rm Im}(\Psi(\tfrac{1-2\la}{4}+it))+{\rm Im}(\Psi(\tfrac{1+2\la}{4}+it))
-{\rm Im}(\Psi(\tfrac{3}{4}+it)-{\rm Im}(\Psi(\tfrac{1}{4}+it)))
\]  
where $\Psi(z)=\Gamma'(z)/\Gamma(z)$ is the digamma function, and using its series expansion $\Psi(z)=-e+
\sum_{n=0}^{\infty}(\frac{1}{n+1}-\frac{1}{n+z})$ with $e$ the Euler constant, we get for $t\geq 0$
\[ \frac{\theta'(t)}{\theta(t)}=-2t\Big(\sum_{n=0}^\infty \frac{1}{(n+\frac{1-2\la}{4})^2+t^2}+\frac{1}{(n+\frac{1+2\la}{4})^2+t^2}-\frac{1}{(n+\frac{3}{4})^2+t^2}-\frac{1}{(n+\frac{1}{4})^2+t^2}\Big)\]
is negative.
Similarly,  the function $\rho: t\mapsto \frac{\Gamma(\frac{1-2\la}{4}-t)\Gamma(\frac{1-2\la}{4}+t)\Gamma(\frac{1+2\la}{4}+t)\Gamma(\frac{1+2\la}{4}-t)}{\Gamma(\frac{3}{4}-t)\Gamma(\frac{1}{4}-t)\Gamma(\frac{3}{4}+t)\Gamma(\frac{1}{4}+t)}$ is increasing in $t\in [0,\tfrac{1-2\la}{4})$: indeed, one has 
\[ \begin{split}
\frac{\rho'(t)}{\rho(t)}=& \Psi(\tfrac{1-2\la}{4}+t)-\Psi(\tfrac{1-2\la}{4}-t)+\Psi(\tfrac{1+2\la}{4}+t)-\Psi(\tfrac{1+2\la}{4}-t)\\
&+ \Psi(\tfrac{3}{4}-t)-\Psi(\tfrac{3}{4}+t)+
\Psi(\tfrac{1}{4}-t)-\Psi(\tfrac{1}{4}+t)
\end{split}\]  
thus using the series expansion of $\Psi(z)$ and factorizing we obtain for $t\in [0,1/4)$
\[ \frac{\rho'(t)}{\rho(t)}=2t\sum_{n=0}^\infty 
\frac{(n+\tfrac{3}{4})^2-(n+\tfrac{1+2\la}{4})^2}
{((n+\tfrac{3}{4})^2-t^2)((n+\tfrac{1+2\la}{4})^2-t^2)}+
\frac{(n+\tfrac{1}{4})^2-(n+\tfrac{1-2\la}{4})^2}
{((n+\tfrac{1}{4})^2-t^2)((n+\tfrac{1-2\la}{4})^2-t^2)}
\geq 0.
\] 
From the spectral theorem, this implies that the $L^2(M)\to L^2(M)$ norm of 
$\Pi_0^\la$ is obtained by evaluating $H_\la$ at either $0$ if $\delta_{\Gamma}\leq 1/2$ or at $i(\delta_{\Gamma}-\demi)$ if $\delta_{\Gamma}>1/2$. This ends the proof.
\end{proof}

\begin{rem} 
Using a similar analysis, the decomposition of $L^2(\Gamma\backslash {\rm SL}_2(\rr))$
into irreducible representations when $\Gamma$ is co-compact 
(see for instance \cite{La,AnZe}) proves that the operator $\Pi_0$ defined in \cite{Gu1} 
is given on a compact hyperbolic manifold $M=\Gamma\backslash \hh^2$ by 
the same formula as in Lemma \ref{Pi0hyp}, but setting $\la=0$ and replacing $\Delta_M$ by 
$\Delta_M(1- P_0)$ if $P_0$ is the orthogonal projection onto $\ker \Delta_M$ in $L^2(M)$.  
\end{rem}

\end{document}